\documentclass[12pt,a4paper]{article}
\pdfoutput=1
\usepackage[T1]{fontenc}
\usepackage[utf8]{inputenc}
\usepackage[english]{babel}
\usepackage{amsmath}
\usepackage{cite}
\usepackage{amsfonts}
\usepackage{amssymb}
\usepackage{amsthm}
\usepackage{mathtools}
\mathtoolsset{showonlyrefs=true}
\usepackage[colorlinks=true,allcolors=blue]{hyperref}
\usepackage[left=2.5cm,right=2.5cm,top=2.5cm,bottom=2.5cm]{geometry}
\usepackage{float}
\usepackage{authblk}
\usepackage{enumerate}
\usepackage{physics}
\usepackage{subcaption}

\newtheorem{thm}{Theorem}[section]
\newtheorem{definition}[thm]{Definition}
\newtheorem{lemma}[thm]{Lemma}

\newtheorem{rem}[thm]{Remark}

\DeclareMathOperator*{\essup}{ess\,sup}
\usepackage{tikz}
\tikzstyle{complex}=[shape=rectangle, draw=black]
\tikzstyle{compartment}=[shape=circle, draw=black, fill=black]

\let\epsilon\varepsilon

\author[1]{Mihály A. Vághy}
\author[1,2]{G\'abor Szederk\'enyi} 
\affil[1]{\small P\'azm\'any P\'eter Catholic University, Faculty of Information Technology and Bionics, Práter u. 50/a,  H-1083 Budapest, Hungary} 
\affil[2]{\small Systems and Control Laboratory, Institute for Computer Science and Control (SZTAKI), Kende u. 13-17, H-1111 Budapest, Hungary}
\title{Persistence and stability of generalized ribosome flow models with time-varying transition rates}
\date{}

\begin{document}
\maketitle
\begin{abstract}
	In this paper the qualitative dynamical properties of so-called generalized ribosome flow models are studied. Ribosome flow models known from the literature are generalized by allowing an arbitrary directed network structure between the compartments and secondly, by assuming a general time-varying rate function describing the compartmental transitions. Persistence of the dynamics is shown using the chemical reaction network (CRN) representation of the system. We show the stability of different compartmental structures including strongly connected ones with an entropy-like logarithmic Lyapunov function. The L1 contractivity of solutions is also studied in the case of periodic reaction rates having the same period. It is also shown that different Lyapunov functions may be assigned to the same model depending on the factorization of the reaction rates.
\end{abstract}
\section{Introduction}
Compartmental models are used to describe and analyze the transport between different containers, called compartments in various natural and technological systems \cite{Haddad2010,Godfrey1983}. Compartments can be assigned to tissues or organs in pharmacokinetic models, mass containers in process systems, distinct disease states in epidemiological models, road sections in transportation systems or different habitats in ecological models. The modeled objects (molecules, people, vehicles, etc.) can move between compartments obeying the given constraints such as limits of directions, flow rates, or capacities. A fundamental feature of compartmental models is that each modeled object can be present in exactly one compartment at a given time. Naturally, compartmental models written in the original physical coordinates belong to the class of nonnegative systems for which the nonnegative orthant is invariant with respect to the dynamics \cite{Farina2000,Rantzer2018}. This special property supports the dynamical analysis and control design in several ways. The controllability, observability, realizability and identifiability of mainly linear compartmental system are addressed in \cite{Brown1980}. An excellent overview of the qualitative dynamical properties of general compartmental systems can be found in \cite{Jacquez1993}. 

The dynamical modeling of the mRNA translation process has been in the focus of research since the second half of the 20th century (see, e.g. \cite{MacDonald1968,Heinrich1980,vonderHaar2012}). The first large scale analysis of gene translation through the so-called ribosome flow model (RFM) was presented in \cite{Reuveni2011}, where the applied second order nonnegative and nonlinear model based on the principle of Totally Asymmetric Exclusion \cite{Shaw2003} was able to capture the most important dynamical features of the translation process. In \cite{Margaliot2013} the RFM was equipped with an appropriate input-output pair, and it was shown that after applying an affine positive output feedback, the system had a unique equilibrium point which is globally stable in the bounded operating domain. A circular RFM structure was analyzed in \cite{Raveh2015}, where the authors proved using the theory of cooperative systems that the system has a continuum of equilibria, but each equilibrium is globally asymptotically stable within the equivalence classes of trajectories determined by the initial conditions. The stability of periodic solutions was also shown. In \cite{Raveh2016} a bounded pool of free ribosomes was added to the RFM generating a competition among the arbitrary number of mRNA molecules for ribosomes. This generates a special network structure for RFM subsystems, for which the uniqueness and stability of equilibria together with the properties of periodic solutions were also shown. Different compartment sizes of the RFM were assumed in \cite{Bar-Shalom2020}, and it was shown that this modification does not change the favorable dynamical properties of the system. In \cite{Jain2022}, the ribosome flow model with Langmuir kinetics (RFMLK) is introduced, and a network structure is constructed with RFMLK subsystems connected through a pool. Among other results, it is shown that the trajectories of such a network always converge to a unique equilibrium. 

Chemical reaction networks (CRNs) also called kinetic systems can be considered as universal descriptors of nonlinear dynamics, especially that of nonnegative systems \cite{Erdi1989}. Since the 1970's the theory of CRNs has been intensively studied, and there are several fundamental results on the relation between network structure/parametrization and dynamical properties \cite{Feinberg2019}. The stability of mass-action type CRNs is most often analyzed using an entropy-like logarithmic Lyapunov function, originally called a ``pseudo-Helmholtz function" in \cite{Horn1972}. Probably the most well-known conjecture of chemical reaction network theory is the ``Global attractor conjecture" according to which complex balanced kinetic systems are globally stable with respect to the nonnegative orthant with the logarithmic Lyapunov function \cite{Craciun2015}. This conjecture was proved for complex balanced reaction networks with a reaction graph of one component \cite{Anderson2011}. One of the most important results from the point of view of this paper is \cite{Chaves2005} studying zero deficiency networks, where the allowed kinetics is more general than mass action, the rate coefficients can be time-varying, and the logarithmic Lyapunov function is also generalized. The Lyapunov-function-based stability analysis of RFMs is mentioned as an important open problem in \cite{Margaliot2012}, which will be addressed in this paper using the CRN representation of the system.

It is interesting to mention that mathematical models which are equivalent to RFMs can also be obtained through a special finite volume spatial discretization of widely used flow models in PDE form \cite{Liptak2021}. These models also have a transparent representation in CRN form supporting further dynamical analysis. An arbitrary directed graph structure of such models with general time-invariant kinetics was considered in \cite{szederkenyi2022persistence}, where the existence and uniqueness of equlibria, persistence and contractivity (non-expansive property) of the solutions was shown using the theory of Petri nets, compartmental systems, and earlier results on RFMs. The stability of this model class with logarithmic Lyapunov functions was shown in \cite{Vaghy2022b}, while a port-Hamiltonian description was given in \cite{Vaghy2022}.

Based on the above overview, the aim of this paper is to extend the results of \cite{szederkenyi2022persistence} and \cite{Vaghy2022b} in the following respects: considering even more general kinetics with explicit time-dependence, the qualitative analysis of periodic solutions, and finally, stability analysis with a family of different logarithmic Lyapunov functions. 

The structure of the paper is the following. Section \ref{sec:not} contains the applied mathematical notations for compartmental models and kinetic systems. In Section \ref{sec:kin} the kinetic representation of the studied model class is described, while new results on persistence and periodic behaviour in the time-varying case are proposed in Section \ref{sec:qua}. Stability analysis results with a family of non-unique logarithmic Lyapunov functions are described in Section \ref{sec:sta}, and finally, Section \ref{sec:con} summarizes the main results of the paper.

\section{Notations and background}\label{sec:not}
In this section, we describe the basic notations and building blocks of a compartmental system class and chemical reaction networks (CRNs). The notations and overview in this section are based on \cite{szederkenyi2022persistence} and \cite{Vaghy2022}. 

\subsection{Compartmental models}
Throughout the paper we consider systems containing a set of interconnected compartments and objects (such as ribosomes, particles, molecules, vehicles etc.) moving between them. We assume that the rate of transfer between compartments depends on the amount of objects in the source compartment as well as on the amount of free space in the target compartment. This naturally implies that each compartment has a well-defined finite \textit{capacity} that limits the amount of modeled quantities that can be contained in the given compartment. We also allow explicit time dependence and in some cases dependence on the amount of objects and free space in other compartments.

For the formal definition, let us consider the set $Q=\qty{q_1,q_2,\dots,q_m}$ of compartments and the set $A\subset Q\times Q$ of transitions, where $(q_i,q_j)\in A$ represents the transition from compartment $q_i$ into $q_j$. Then, the directed graph $D=(Q,A)$ is called the \textit{compartmental graph} and it describes the structure of the compartmental model. The transitions are assumed to be immediate, thus loop edges are not allowed in the model since they do not introduce additional dynamical terms. Similarly, we do not allow parallel edges between two compartments in the same direction since they can be replaced by a single transition. We say that a (compartmental) graph is \textit{strongly connected} if there exists a directed path between any two vertices in both directions, and we say that a graph is \textit{weakly reversible} if it is a collection of isolated strongly connected subgraphs.

For each compartment $q_i$ we introduce the sets of \textit{donors} and \textit{receptors}, respectively, as
\begin{equation}
	\begin{aligned}
		\mathcal{D}_i&=\qty\big{j\in\qty{1,2,\dots,m}\big|(q_j,q_i)\in A},\\
		\mathcal{R}_i&=\qty\big{j\in\qty{1,2,\dots,m}\big|(q_i,q_j)\in A};
	\end{aligned}
\end{equation}
that is, the set of donors of a given compartment are the compartments where an incoming transition originates from and the set of receptors are the compartments where an outgoing transition terminates in.

\subsection{Chemical reaction networks (kinetic systems)}\label{subsec:CRNs}
In this subsection we give a brief introduction of kinetic systems based on \cite{Feinberg2019,Horn1972}, where more details can be found. A chemical reaction network (CRN) contains a set of \textit{species} $\Sigma=\qty{X_1,X_2,\dots,X_N}$ and the corresponding species vector is given by $X=[X_1~X_2~\dots~X_N]^{\mathrm{T}}$. The species of a CRN are transformed into each other through \textit{elementary reaction steps} of the form
\begin{equation}
	C_j\xrightarrow{\mathcal{K}_j(t)}C_{j'}\qquad j=1,2,\dots,R,
\end{equation}
where $C_j=y_j^{\mathrm{T}}X$ and $C_{j'}=y_{j'}^{\mathrm{T}}X$ are the source and product \textit{complexes}, respectively, the vectors $y_j,y_{j'}\in\mathbb{N}_0^N$ are \textit{stoichiometric coefficient vectors} and functions $\mathcal{K}_j:\overline{\mathbb{R}}_+^N\times\overline{\mathbb{R}}_+\mapsto\overline{\mathbb{R}}_+$ are the \textit{rate functions} with $\overline{\mathbb{R}}_+$ denoting the set of nonnegative real numbers. The matrix $Y$ containing the stoichiometric coefficient vectors as columns is called the \textit{stoichiometric matrix}. The subspace $\mathcal{S}\subset\mathbb{R}^N$ spanned by the so-called \textit{reaction vectors} $y_{j'}-y_j$ is called the \textit{stoichiometric subspace} of the CRN.

The CRN structure can be uniquely described by a directed graph as follows. For each complex we assign a vertex in the graph and and for each elementary reaction step of the form $C_j\rightarrow C_{j'}$ we assign a directed edge between the corresponding vertices. We call the resulting graph the \textit{reaction graph} of the CRN. The \textit{deficiency} of the CRN is defined as $\delta=m-\ell-s$, where $m$ is the number of distinct complexes, $\ell$ is the number of linkage classes (graph components) in the reaction graph and $s$ is the dimension of the stoichiometric subspace.

Let $x(t)\in\overline{\mathbb{R}}_+^N$ denote the state vector of the species as a function of time for $t\ge0$. Based on the above, the dynamics of the CRN is given by
\begin{equation}\label{eq:kin}
	\dot x(t)=\sum_{j=1}^R\mathcal{K}_j(x,t)[y_{j'}-y_j].
\end{equation}

We assume that a reaction can only take place if each species of the given reaction have nonzero concentration; that is, we assume that $\mathcal{K}_j\qty\big(x(t),t)=0$ whenever there exists $k\in\mathrm{supp}(y_j)$ such that $x_k(t)=0$, where we say that $k\in\mathrm{supp}(y_j)$ if $[y_j]_k>0$. This property ensures the invariance of the nonnegative orthant (or a part of it). We also presume standard regularity assumptions of the rate functions that guarantee local existence and uniqueness of solutions. Different results in this paper require different sets of such assumptions, thus for the sake of generality they will be specified later. Dynamics of the form of \eqref{eq:kin} is called \textit{persistent} if no trajectory that starts in the positive orthant has an omega-limit point on the boundary of $\mathbb{R}_+^N$.

We note that for any $v\in\mathcal{S}^{\perp}$ (where $\mathcal{S}$ denotes the stoichiometric subspace) we have that
\begin{equation}
	\langle \dot x,v\rangle=\sum_{j=1}^R\mathcal{K}_j(x,t)\langle y_{j'}-y_j,v\rangle=0
\end{equation}
and thus $\langle x,v\rangle$ is constant. Since $v\in\mathcal{S}^{\perp}$ was arbitrary we have that $x(t)\in x(0)+\mathcal{S}$. This shows that the translates of $\mathcal{S}$ define invariant linear manifolds for the system. We further define for each $p\in\mathbb{R}_+^n$ a positive stoichiometric compatibility class $\mathcal{S}_p=(p+\mathcal{S})\cap\overline{\mathbb{R}}_+^n$.

A set of ODEs of the form $\dot x=f(x,t)$ is called kinetic if it can be written in the form \eqref{eq:kin} with appropriate rate functions and stoichiometric coefficient vectors. 

\section{Kinetic representation}\label{sec:kin}
In this section we construct a kinetic representation of the above compartmental system class. To do so, we assign a CRN that incorporates the compartmental structure. This allows the introduction of a system of ODEs of the form \eqref{eq:kin} describing the time evolution of the compartmental model. Some of the following steps are described in \cite{szederkenyi2022persistence} or \cite{Vaghy2022} in a time-invariant setting but here we recall and extend them for convenience.

\subsection{Kinetic modelling of compartmental transitions}
Let us consider a compartmental model $D=(Q,A)$. Let the set of species be $\Sigma=\qty{N_1,N_2,\dots,N_m}\cup\qty{S_1,S_2,\dots,S_m}$ where $N_i$ and $S_i$ represent the number of particles and available spaces in compartment $q_i$, respectively. To each transition $(q_i,q_j)\in A$ we assign a reaction of the form
\begin{equation}
	N_i+S_j\xrightarrow{\mathcal{K}_{ij}}N_j+S_i,
\end{equation}
where $\mathcal{K}_{ij}$ is the rate function of the transition. Such a reaction represents that during the transition from compartment $q_i$ to compartment $q_j$ the number of items decreases in $q_i$ and increases in $q_j$, while the number of available spaces increases in $q_i$ and decreases in $q_j$. Let $n_i$ and $s_i$ denote the continuous amount of particles and free space in $q_i$, respectively.

Based on \eqref{eq:kin} the dynamics of the system is given by
\begin{equation}
	\begin{aligned}
		\dot n_i&=\sum_{j\in\mathcal{D}_i}\mathcal{K}_{ji}(n,s,t)-\sum_{j\in\mathcal{R}_i}\mathcal{K}_{ij}(n,s,t),\label{eq:ns}\\
		\dot s_i&=-\sum_{j\in\mathcal{D}_i}\mathcal{K}_{ji}(n,s,t)+\sum_{j\in\mathcal{R}_i}\mathcal{K}_{ij}(n,s,t)
	\end{aligned}
\end{equation}
where $n$ and $s$ denote the vectorized form of the variables $n_i$ and $s_i$, respectively. It is easy to check that the model class in Eq. \eqref{eq:ns} contains ribosome flow models described in \cite{Margaliot2012} or \cite{Bar-Shalom2020}, and extends them in two ways: firstly, the reaction rate function $\mathcal{K}$ is not necessarily mass-action type and moreover, is time-varying, and secondly, the compartmental graph of the system can be arbitrary (i.e., there can be particle transition between any two compartments). Note, that we also allow the transition rates to depend on the amount of objects and free space in other compartments as well, perhaps describing inhibitory phenomena. Therefore, we call \eqref{eq:ns} a \textit{generalized time-varying ribosome flow model}.
\bigskip

Clearly the reaction graph of the assigned CRN of a compartmental model is generally not strongly connected nor weakly reversible even if the compartmental graph is strongly connected. In fact, the reaction graph is weakly reversible if and only if each transition in the compartmental system is reversible. Even though the reaction graph, in some sense, loses the regularities of the compartmental graph, we can explicitly determine its deficiency from the compartmental topology and, as described in \cite{szederkenyi2022persistence}, CRNs of the form \eqref{eq:ns} exhibit persistence and stability properties in various senses in the time-invariant case.

\subsection{Deficiency of CRNs realizing compartmental models}
For a compartmental system $D=(Q,A)$ let $|D|=\qty\big(Q,\tilde A)$ denote the undirected graph where the parallel edges are merged. 

\begin{thm}\label{thm:deficiency}
	The deficiency of a CRN assigned to a compartmental model $D=(Q,A)$ is equal to the number of chordless cycles in the undirected graph $|D|=\qty\big(Q,\tilde A)$.
\end{thm}
\begin{proof}
	For each transition between $q_i$ and $q_j$ we assign two complexes, namely $N_i+S_j$ and $S_i+N_j$, regardless of the transitions' direction, so reversible reactions do not introduce additional complexes, and thus the number of stoichiometrically distinct complexes is $M=2|\tilde A|$. A complex of the form $N_i+S_j$ is only connected with the complex $S_i+N_j$, and thus we have $\ell=|\tilde A|$ linkage classes each consisting of exactly two complexes. To find the dimension of the stoichiometric subspace, denoted by $S=\dim\mathcal{S}$, observe that the reaction vector of a reaction of the form $N_i+S_j\rightarrow N_j+S_i$ is
	\begin{equation}\label{eq:y_ij}
		y_{i\rightarrow j}=-e_i+e_j+e_{m+i}-e_{m_j},
	\end{equation}
	where $e_k\in\mathbb{R}^{2m}$ denotes the $k$th unit vector. Again, since $y_{i\rightarrow j}=-y_{j\rightarrow i}$ it suffices to consider the undirected graph $|D|$. Assume that $y_{i\rightarrow j}$ is such that
	\begin{equation}
		y_{i\rightarrow j}=\sum c_{l\rightarrow l'}y_{l\rightarrow l'}.
	\end{equation}
	Then by \eqref{eq:y_ij} we have that for each non-zero term of the form $c_{.\rightarrow l'}y_{.\rightarrow l'}$ the right-hand side also contains at least one non-zero term $c_{l'\rightarrow .}y_{l'\rightarrow.}$, including the terms $c_{i\rightarrow.}y_{i\rightarrow.}$ and $c_{.\rightarrow j}y_{.\rightarrow j}$. This shows that the edges corresponding to the reaction vectors of the right-hand side form possibly multiple cycles in $|D|$. Without the loss of generality we may assume that this subgraph does not contain cycles isolated from $(q_i,q_j)$. We have to consider the following cases:
	\begin{enumerate}
		\item First, we assume that the right-hand side is a single chordless cycle and contains the transitions
			\begin{equation}
				q_i\rightarrow q_{l_1}\rightarrow q_{l_2}\rightarrow\dots\rightarrow q_{l_r}\rightarrow q_j\rightarrow q_i.
			\end{equation}
			Taking the inner product of unit vectors $e_i,e_{l_1},e_{l_2},\dots,e_{l_r},e_j$ and
			\begin{equation}
				y_{i\rightarrow j}=c_{i\rightarrow l_1}y_{i\rightarrow l_1}+\sum_{k=1}^{r-1}c_{l_k\rightarrow l_{k+1}}y_{l_k\rightarrow l_{k+1}}+c_{l_r\rightarrow j}y_{l_r\rightarrow j}
			\end{equation}
			yields the system of linear equations:
			\begin{equation}
				\begin{aligned}
					-1&=-c_{i\rightarrow l_1}\\
					0&=c_{i\rightarrow l_1}-c_{l_1\rightarrow l_2}\\
					0&=c_{l_1\rightarrow l_2}-c_{l_2\rightarrow l_3}\\
					&\vdots\\
					0&=c_{l_{r-1}\rightarrow l_r}-c_{l_r\rightarrow j}\\
					1&=c_{l_r\rightarrow j}
				\end{aligned}
			\end{equation}
			which clearly has one solution where each weight is equal to one.
		\item If the right-hand side consists of multiple cycles, then repeatedly using the previous argument we can replace the arcs not containing $(q_i,q_j)$ with chords. Note, that if the reaction vector corresponding to the chord is already on the right-hand side, then we just have to modify its coefficient. This method decomposes the right-hand side and will leave us with one chordless cycle containing $(q_i,q_j)$, leading back to the previous case with exactly one solution. Repeating the arc substitutions we can see that each arc becomes a chordless cycle with the reintroduced edges and the arising systems of linear equations have exactly one solution.
	\end{enumerate}
	The first case above shows that the dimension of the stiochiometric subspace reduces by one for each set of reaction vectors that correspond to edges forming a chordless cycle in $|D|$ and the second case shows that is reduced by that exact amount. If $\sigma$ denotes the number of chordless cycles in $\tilde Q$, then the deficiency of the reaction network can be computed as $\delta=M-\ell-S=2|\tilde A|-|\tilde A|-\qty\big(|\tilde A|-\sigma)=\sigma$.
\end{proof}

\subsection{Linear conservation laws}\label{sec:conservation}
System \eqref{eq:ns} exhibits conservation in several senses. First of all, we have that
\begin{equation}
	\sum_{i=1}^m\qty\big(\dot n_i+\dot s_i)=0,
\end{equation}
thus the sum of modeled quantities and free spaces in the system is constant along the trajectories of \eqref{eq:ns}; that is, the function $H:\mathbb{R}^{2m}\mapsto\mathbb{R}$ defined for $x\in\mathbb{R}^{2m}$ as
\begin{equation}\label{eq:H}
	H(x)=\sum_{i=1}^{2m}x_i,
\end{equation}
is a first integral, where $x_1,x_2,\dots,x_m$ and $x_{m+1},x_{m+2},\dots,x_{2m}$ correspond to the variables $n_1,n_2,\dots,n_m$ and $s_1,s_2,\dots,s_m$, respectively. Our next observation is that $\dot n_i+\dot s_i=0$ holds for each compartment, thus $c_i:=n_i+s_i$ is the constant capacity of compartment $q_i$. Let $c^{(m)}$ be a vector such that its $i$th coordinate is $c_i$. Substituting $s=c^{(m)}-n$ we can rewrite \eqref{eq:ns} in a reduced state space as
\begin{equation}
	\begin{aligned}
		\dot n_i&=\sum_{j\in\mathcal{D}_i}\mathcal{K}_{ji}\qty\big(n,c^{(m)}-n,t)-\sum_{j\in\mathcal{R}_i}\mathcal{K}_{ij}\qty\big(n,c^{(m)}-n,t)\label{eq:n}
	\end{aligned}
\end{equation}
or after an analogous substitution, as 
\begin{equation}
	\begin{aligned}
		\dot s_i&=-\sum_{j\in\mathcal{D}_i}\mathcal{K}_{ji}\qty\big(c^{(m)}-s,s,t)+\sum_{j\in\mathcal{R}_i}\mathcal{K}_{ij}\qty\big(c^{(m)}-s,s,t).\label{eq:s}
	\end{aligned}
\end{equation}
As a consequence of the preceding observations, the function $\tilde H:\mathbb{R}^m\mapsto\mathbb{R}$, defined for $x\in\mathbb{R}^m$ as
\begin{equation}\label{eq:tH}
	\tilde H(x)=\sum_{i=1}^mx_i
\end{equation}
is a first integral for \eqref{eq:n}, in which case each $x_i=n_i$ (and similarly for \eqref{eq:s} if each $x_i=s_i$). This shows that while the state space of the decomposed systems is $\tilde C:=[0,c_1]\times[0,c_2]\times\dots\times[0,c_m]$, for a given initial condition $x(0)\tilde C$ the trajectories are contained in the $(m-1)$-dimensional manifold (hyperplane) defined by
\begin{equation}
	\qty\big{x\in\tilde C\big|\tilde H(x)-\tilde H\qty\big(x(0))=0}.
\end{equation}
For a generalized ribosome flow define $c=\sum_{i=1}^nc_i$ and for $r\in[0,c]$ let $L_r\subset\tilde C$ be the level set of $H$ corresponding to $r$; that is,
\begin{equation}
	L_r=\qty\big{a\in\tilde C:H(a)=r}.
\end{equation}

\subsubsection*{Example 1.1}\label{ex1}
As a small example let us consider the compartmental model given by $D=(Q,A)$, where
\begin{equation}
	\begin{aligned}
		Q&=\qty{q_1,q_2,q_2},\\
		A&=\qty\big{(q_1,q_2),(q_2,q_3),(q_3,q_1)}.
	\end{aligned}
\end{equation}
The topology is shown in Figure \ref{fig:comp}. 
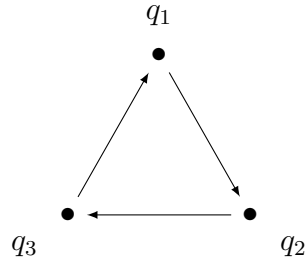
\begin{figure}[H]
	\begin{center}
		\begin{tikzpicture}[scale=0.8, every node/.style={scale=1}]
			\node[label=$q_1$] (q1) at (0,0) {$\bullet$};
			\node[label={330:$q_2$}] (q2) at (0.5*3,-0.886*3) {$\bullet$};
			\node[label={210:$q_3$}] (q3) at (-0.5*3,-0.886*3) {$\bullet$};
			\draw[-latex] (q1) -- (q2);
			\draw[-latex] (q2) -- (q3);
			\draw[-latex] (q3) -- (q1);
		\end{tikzpicture}
	\end{center}
	\caption{Compartmental graph of a triangular model}\label{fig:comp}
\end{figure}
The corresponding CRN has the following species and reactions:
\begin{equation}
	\begin{aligned}
		&\Sigma=\qty{N_1,N_2,N_3,S_1,S_2,S_3}\\
		&R_1:N_1+S_2\xrightarrow{\mathcal{K}_{12}}S_1+N_2\\
		&R_2:N_2+S_3\xrightarrow{\mathcal{K}_{23}}S_2+N_3\\
		&R_3:N_3+S_1\xrightarrow{\mathcal{K}_{31}}S_3+N_1.
	\end{aligned}
\end{equation}
It is easy to see that, indeed, the reaction graph is not weakly reversible and its deficiency is one. The dynamics of the model in the full state space is given by \eqref{eq:ns} as
\begin{equation}
	\begin{aligned}
		\dot n_1&=\mathcal{K}_{31}(n,s,t)-\mathcal{K}_{12}(n,s,t)\\
		\dot s_1&=-\mathcal{K}_{31}(n,s,t)+\mathcal{K}_{12}(n,s,t)\\
		\dot n_2&=\mathcal{K}_{12}(n,s,t)-\mathcal{K}_{23}(n,s,t)\\
		\dot s_2&=-\mathcal{K}_{12}(n,s,t)+\mathcal{K}_{23}(n,s,t)\\
		\dot n_3&=\mathcal{K}_{23}(n,s,t)-\mathcal{K}_{31}(n,s,t)\\
		\dot s_3&=-\mathcal{K}_{23}(n,s,t)+\mathcal{K}_{31}(n,s,t)\\
	\end{aligned}
\end{equation}
which can be rewritten in the reduced state space based on \eqref{eq:n} as
\begin{equation}
	\begin{aligned}
		\dot n_1&=\mathcal{K}_{31}\qty\big(n,c^{(m)}-n,t)-\mathcal{K}_{12}\qty\big(n,c^{(m)}-n,t)\\
		\dot n_2&=\mathcal{K}_{12}\qty\big(n,c^{(m)}-n,t)-\mathcal{K}_{23}\qty\big(n,c^{(m)}-n,t)\\
		\dot n_3&=\mathcal{K}_{23}\qty\big(n,c^{(m)}-n,t)-\mathcal{K}_{31}\qty\big(n,c^{(m)}-n,t).
	\end{aligned}
\end{equation}

\section{Qualitative dynamical analysis}\label{sec:qua}
In this section we show that systems of the form \eqref{eq:ns} exhibit various interesting dynamical properties that can be characterized under different assumptions of the transition rate functions. First we will consider time-invariant systems to demonstrate the regularity of equilibria. Then we return to time-varying systems to generalize the results of \cite{szederkenyi2022persistence}.

\subsection{Equilibria of time-invariant systems}
In this subsection we assume that the $\mathcal{K}_{ij}(n,s,t)$ rate functions are continuously differentiable and only depend on the variables $n_i$ and $s_j$ in a nondecreasing manner; that is, we assume that $\mathcal{K}_{ij}(n,s,t)\equiv\mathcal{K}_{ij}(n_i,s_j)$ for each $i$ and $j$. Then the results \cite[Propositions 5.5, 5.6]{szederkenyi2022persistence} show that a system of the form \eqref{eq:n} is cooperative (the name also highlights the importance of the exclusion of inhibitory phenomena), is (strongly) monotone and each level set $L_r$ contains a unique globally (relative to its level set) asymptotically stable steady state. This implies that the steady states form a linearly ordered set. For $i=1,2,\dots,m$ let $e_i:[0,c]\mapsto[0,c_i]$ denote the $i$th coordinate function of the steady state; that is, let
\begin{equation}
	e_i(r):=\lim_{t\rightarrow\infty}\rho\qty\big(t,n(0))_i
\end{equation}
where $n(0)\in L_r$ is arbitrary and $\rho\qty\big(t,n(0))$ denotes the solution at time $t$ with $\rho\qty\big(0,n(0))=n(0)$. Clearly each $e_i$ is continuous and the monotonicity of the system also shows that each $e_i$ function is strictly increasing; that is, they are differentiable almost everywhere and their derivative are positive.

\subsubsection*{Example 1.2}
Let us consider the triangular compartmental model in Figure \ref{fig:comp}. Its time evolution in the reduced state space is given in the form \eqref{eq:n} as
\begin{equation}
	\begin{aligned}
		&\dot n_1=\mathcal{K}_{31}(n_3,c_1-n_1)-\mathcal{K}_{12}(n_1,c_2-n_2),\\
		&\dot n_2=\mathcal{K}_{12}(n_1,c_2-n_2)-\mathcal{K}_{23}(n_2,c_3-n_3),\\
		&\dot n_3=\mathcal{K}_{23}(n_2,c_3-n_3)-\mathcal{K}_{31}(n_3,c_1-n_1),\\
	\end{aligned}
\end{equation}
and for the simulations we set capacities $c_1=5$, $c_2=25$, $c_3=50$. The rate functions in the different cases are assumed to have the form $\mathcal{K}_{ij}(n_i,c_j-n_j)=k_{ij}n_i(c_j-n_j)$ (corresponding to mass-action kinetics) or to be rational functions of the form 
\begin{equation*}
	\mathcal{K}_{ij}(n_i,c_j-n_j)=k_{ij}\frac{n_i^3}{(l+n_i)^3}\cdot\frac{(c_j-n_j)^3}{(l+c_j-n_j)^3}
\end{equation*}
for some $l>0$ with $k_{12}=100$, $k_{23}=40$, $k_{31}=60$. Figure \ref{fig:equilibria} shows the equilibrium curves for these rate functions with various $l$ values.

\begin{figure}[H]
	\begin{center}
		\includegraphics[scale=0.8]{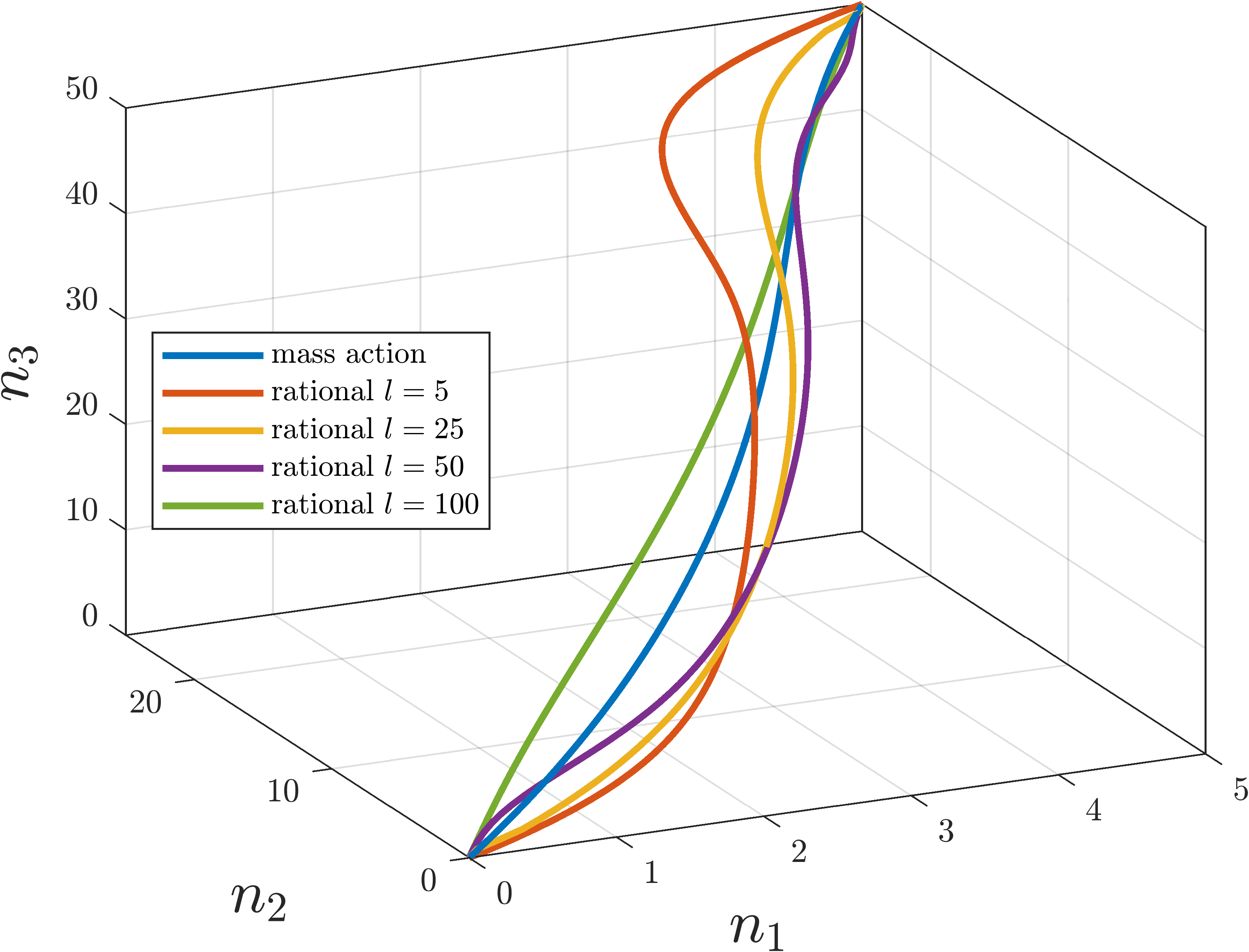}
	\end{center}
	\caption{Loci of equilibria of a triangular model as a function of total amount of modeled quantities for different $l$ saturation parameters}\label{fig:equilibria}
\end{figure}

\subsubsection*{Example 2}\label{ex:NSR}
Let us consider consider a not strongly connected compartmental model given by $D=(Q,A)$, where
\begin{equation}
	\begin{aligned}
		Q&=\qty{q_1,q_2,q_2},\\
		A&=\qty\big{(q_2,q_3),(q_3,q_2),(q_3,q_1)}.
	\end{aligned}
\end{equation}
The topology is shown in Figure \ref{fig:ncomp}. 
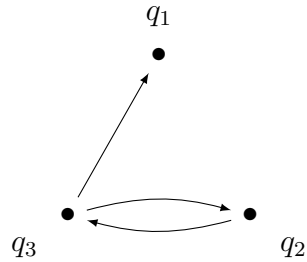
\begin{figure}[H]
	\begin{center}
		\begin{tikzpicture}[scale=0.8, every node/.style={scale=1}]
			\node[label=$q_1$] (q1) at (0,0) {$\bullet$};
			\node[label={330:$q_2$}] (q2) at (0.5*3,-0.886*3) {$\bullet$};
			\node[label={210:$q_3$}] (q3) at (-0.5*3,-0.886*3) {$\bullet$};
			\path[-latex] (q2) edge[bend left=15] (q3);
			\path[-latex] (q3) edge[bend left=15] (q2);
			\draw[-latex] (q3) -- (q1);
		\end{tikzpicture}
	\end{center}
	\caption{Compartmental graph of a not strongly connected model}\label{fig:ncomp}
\end{figure}
The corresponding CRN has the following species and reactions:
\begin{equation}
	\begin{aligned}
		&\Sigma=\qty{N_1,N_2,N_3,S_1,S_2,S_3}\\
		&R_1:N_2+S_3\xrightarrow{\mathcal{K}_{23}}S_2+N_3\\
		&R_2:N_3+S_2\xrightarrow{\mathcal{K}_{32}}S_3+N_2\\
		&R_3:N_3+S_1\xrightarrow{\mathcal{K}_{31}}S_3+N_1.
	\end{aligned}
\end{equation}
The dynamics of the system in the reduced state space is given by
\begin{equation}
	\begin{aligned}
		\dot n_1&=\mathcal{K}_{31}\qty\big(n_3,c_1-n_1)\\
		\dot n_2&=\mathcal{K}_{32}\qty\big(n_3,c_2-n_2)-\mathcal{K}_{23}\qty\big(n_2,c_3-n_3)\\
		\dot n_3&=\mathcal{K}_{23}\qty\big(n_2,c_3-n_3)-\mathcal{K}_{32}\qty\big(n_3,c_2-n_2)-\mathcal{K}_{31}\qty\big(n_3,c_1-n_1).
	\end{aligned}
\end{equation}
Since the compartmental graph is not strongly connected the persistence and stability results of \cite{szederkenyi2022persistence} are not applicable. However, empirical results show that the long-time behaviour of the system still exhibits some regularity, which can be divided into two cases base on the initial values of the system:
\begin{enumerate}
	\item If $r:=H\qty\big(n(0))\le c_1$, then
		\begin{equation*}
			\lim_{t\rightarrow\infty}n_2(t)=\lim_{t\rightarrow\infty}n_3(t)=0\quad\text{and}\quad\lim_{t\rightarrow\infty}n_1(t)=r.
		\end{equation*}
	\item If $r:=H\qty\big(n(0))>c_1$, then
		\begin{equation*}
			\lim_{t\rightarrow\infty}n_1(t)=c_1
		\end{equation*}
		and $n_1(t)$ and $n_2(t)$ will converge to the unique equilibrium on the level set
		\begin{equation*}
			\qty\big{(n_2,n_3)\in[0,c_2]\times[0,c_3]\big|n_2+n_3=r-c_1}
		\end{equation*}
		of the reduced compartmental model $D'=(Q',A')$ given by $Q'=\qty{q_2,q_3}$, $A'=\qty\big{(q_2,q_3),(q_3,q_2)}$. Note that since $D'$ is strongly connected, the results of \cite{szederkenyi2022persistence} and the above investigation can be applied.
\end{enumerate}
For the simulations we set $c_1=c_2=c_3=100$. The rate functions in the different cases are assumed to have form $\mathcal{K}_{ij}(n_i,c_j-n_j)=k_{ij}n_i(c_j-n_j)$ (corresponding to mass-action kinetics) or to be rational functions of the form 
\begin{equation*}
	\mathcal{K}_{ij}(n_i,c_j-n_j)=k_{ij}\frac{n_i}{l+n_i}\cdot\frac{c_j-n_j}{l+c_j-n_j}
\end{equation*}
for some $l>0$ with $k_{23}=15$, $k_{32}=25$, $k_{31}=35$. Figure \ref{fig:NSR_equilibria} shows the equilibrium curves for these rate functions with various $l$ values. As described by the above cases we see that until the sum of the initial value exceed the capacity of the $q_1$ compartment the equilibrium lies on the $n_1$ axis. After that the equilibrium lies on the plane $\qty\big{n_1=c_1}\subset\mathbb{R}^3$ and since $D'$ is strongly connected we have that the coordinate functions of the equilibria $e_2(r)$ and $e_3(r)$, restricted to the set $[c_1,c]$, are continuous and strictly increasing. We note that while the system is not strongly connected it exhibits many similar qualitative properties as strongly connected models. For example, for initial values satisfying $H\qty\big(n(0))>c_1$ the system is Lyapunov stable as described in \cite{Vaghy2022}.

\begin{figure}[H]
	\begin{center}
		\includegraphics[scale=0.8]{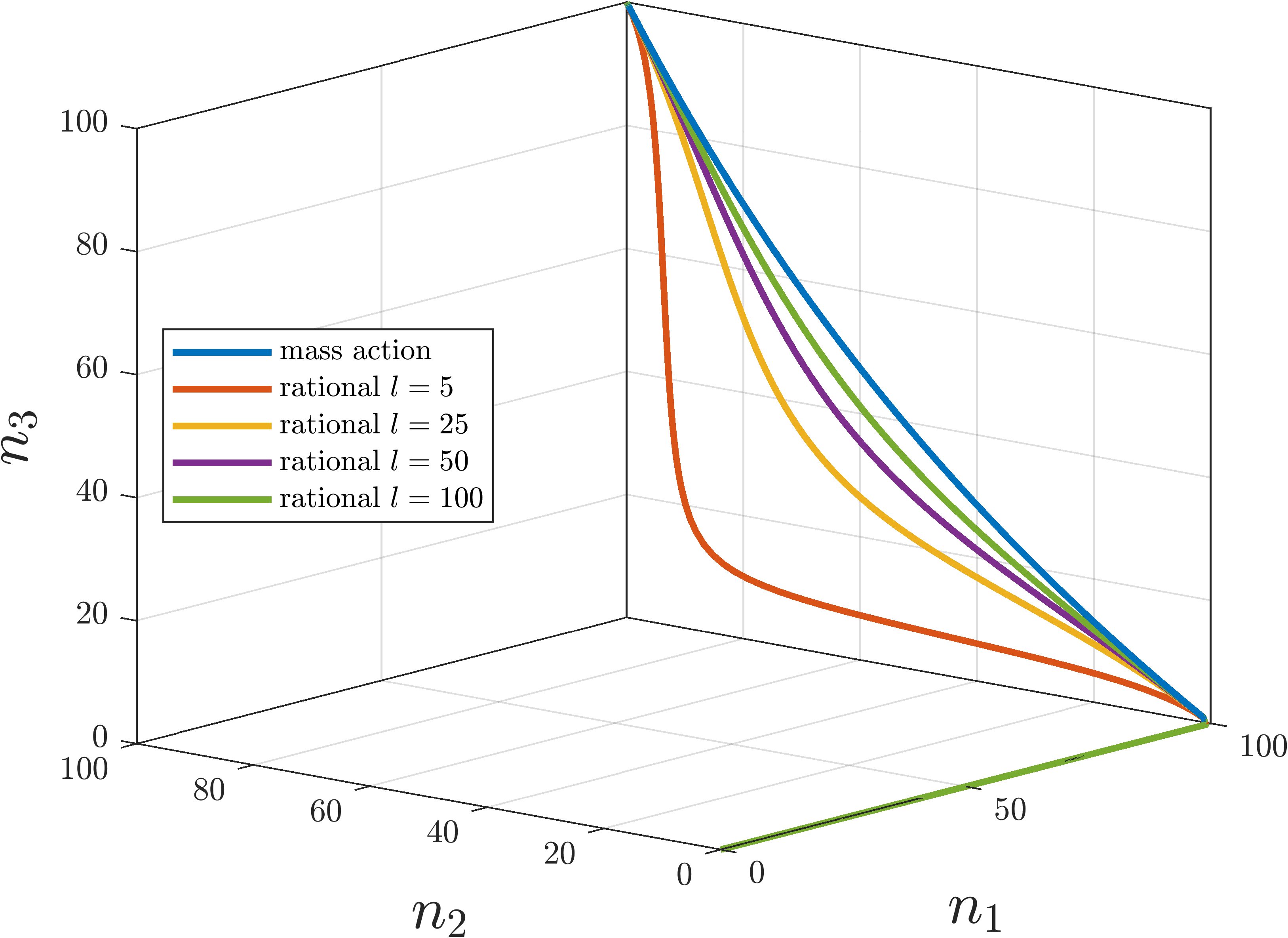}
	\end{center}
	\caption{Loci of equilibria of a not strongly connected model as a function of amount of modeled quantities for different $l$ saturation parameters}\label{fig:NSR_equilibria}
\end{figure}

\begin{rem}
	The authors hypothesize that the long-time behaviour of a compartmental model with \textbf{arbitrary} compartmental structure can be similarly described. Recall that a (compartmental) graph $D=(Q,A)$ can be written as a directed acyclic hypergraph of strongly connected components. The hypergraph will then contain three types of components:
	\begin{enumerate}
		\item we call a component trap if it does not have any outgoing edges,
		\item we call a component source if it does not have any incoming edges,
		\item we call a component intermediate if it is not a trap and not a source.
	\end{enumerate}
	Based on the initial value and the exact compartmental structure the following phenomena can be observed:
	\begin{itemize}
		\item Traps (and only traps) can become full, thus possibly creating new traps.
		\item Sources (and only sources) can become empty, thus possibly creating new sources.
		\item After a sufficient number of traps are filled and sources are emptied, the compartmental graph $D$ is decomposed into isolated strongly connected components; that is, the resulting graph is weakly reversible, in which case the results of \cite{szederkenyi2022persistence} can be applied.
	\end{itemize}
	While these observations are elementary and show that the system is stable, the equilibria are clearly non-unique with respect to the total mass of the network and in general it is not straightforward to predict from the initial value which components will fill and empty.
\end{rem}

\subsection{Persistence}\label{sub:per}
In this subsection we consider time-varying generalized ribosome flows of the form \eqref{eq:ns} only under mild regularity assumptions described by the following theorem, which is based on the results of \cite{Angeli2011} but the statements are rephrased to be more aligned with our framework. For the definition of notions related to Petri nets (e.g. siphons) and their exact connection with CRNs we refer to \cite{Angeli2011,szederkenyi2022persistence}.

\begin{thm}{\cite{Angeli2011}}\label{thm:tv_persistence}
	The dynamics of a CRN of the form \eqref{eq:kin} is persistent if
	\begin{enumerate}[(i)]
		\item Each siphon of the CRN contains a subset of species which define a positive linear conserved quantity for the dynamics.
		\item There exists a positive linear conserved quantity $c^{\mathrm{T}}x$ for the dynamics.
		\item There are nonnegative, continuous functions $\underline{\mathcal{K}}_j(x)$, $\overline{\mathcal{K}}_j(x)$ such that
			\begin{enumerate}
				\item if $x_k>\tilde x_k$ for each $k\in\mathrm{supp}(y_j)$, then $\underline{\mathcal{K}}_j(x)>\underline{\mathcal{K}}_j(\tilde x)$ (and similarly for $\overline{\mathcal{K}}_j$) holds for each $j=1,2,\dots,R$, and
				\item for each $j=1,2,\dots,R$, for all $x\in\mathbb{R}_+^N$ and for all $t\ge0$ we have $\underline{\mathcal{K}}_j(x)\le\mathcal{K}_j(x,t)\le\overline{\mathcal{K}}_j(x)$.
			\end{enumerate}
	\end{enumerate}
\end{thm}

To verify condition (i) we would, in general, need to enumerate all siphons of the CRN, which is well-known to be an NP-hard problem. However, in our recent paper \cite{szederkenyi2022persistence} we explicitly characterized the siphons of a CRN assigned to a strongly connected compartmental models in the time-invariant case. However, one can observe that conditions (i) and (ii) of \ref{thm:tv_persistence} are independent of the choice of transition rates and even independent from whether the system is time-invariant or not; that is, our results, formulated in the following theorem, hold for time-varying compartmental systems as well.
\begin{thm}{\cite[Corollary 4.6]{szederkenyi2022persistence}}
	A siphon in the Petri net of a strongly connected compartmental graph either contains the vertices $N_i$ and $S_i$ corresponding to the same compartment $q_i$, or it contains all the vertices $N_1,N_2,\dots,N_m$ or $S_1,S_2,\dots,S_m$.
\end{thm}
Then the conclusions of Section \ref{sec:conservation} show that conditions (i) and (ii) are satisfied by virtue of the first integrals \eqref{eq:tH} and \eqref{eq:H}, respectively.

It is not straightforward to determine exactly what types of reaction rates satisfy condition (iii). For the sake of specificity, we characterize a class of reaction rates of special interest which can be written in the following form
\begin{equation}\label{eq:K}
	\mathcal{K}_{ij}(n,s,t)=k_{ij}(t)\frac{\theta_i(n_i)\nu_j(s_j)}{1+\Psi_{ij}(n,s)}
\end{equation}
where we assume that the transformations $\theta_i,\nu_j\in\mathcal{C}^1(\mathbb{R})$ are nondecreasing, have $\theta_i(0)=\nu_j(0)=0$ and satisfy $\int_0^1|\log\theta_i(r)|\dd{r}<\infty$ and $\int_0^1|\log\nu_j(r)|\dd{r}<\infty$ for each $i,j=1,2,\dots,m$. We also assume that the functions $\Psi_{ij}$ take the form
\begin{equation}
	\Psi_{ij}(n,s)=\sum\alpha_{r^{(1)},r^{(2)}}\prod_{l=1}^m\theta_l^{r_l^{(1)}}(n_l)\nu_l^{r_l^{(2)}}(s_l)
\end{equation}
where $r^{(1)},r^{(2)}\in\mathbb{N}^m$ and $\alpha_{r^{(1)},r^{(2)}}\in\overline{\mathbb{R}}_+$. We further assume that for $k_{ij}(t)$ there exist $\underline k_{ij},\overline k_{ij}>0$ such that $\underline k_{ij}\le k_{ij}(t)\le\overline k_{ij}$ for all $t\ge0$. In this case we have
\begin{equation}
	\underline{\mathcal{K}}_{ij}(n_i,s_j):=\frac{\underline k_{ij}}{1+\Psi_{ij}(c^{(m)},c^{(m)})}\theta_i(n_i)\nu_j(s_j)\le\mathcal{K}_{ij}(n_i,s_j,t)\le\overline k_{ij}\theta_i(n_i)\nu_j(s_j)=:\overline{\mathcal{K}}_{ij}(n_i,s_j)
\end{equation}
which are clearly monotonous in the sense of Theorem \ref{thm:tv_persistence}, and thus condition (i) is satisfied and the system is persistent.

\begin{rem}\label{rem:eps}
	The above investigation and, in particular, condition (iii) of Theorem \ref{thm:tv_persistence} shows that Lemmata 5.1, 5.2 and Remark 5.3 of \cite{szederkenyi2022persistence} can be modified to the time-varying case; that is, for a system of the form \eqref{eq:ns} with strongly connected compartmental graph and reaction rates of the form \eqref{eq:K}, for each $\tau>0$ there exists $\epsilon(\tau)>0$ with $\epsilon(\tau)\rightarrow0$ as $\tau\rightarrow0$ such that $n_i(t),s_i(t)\in[\epsilon,c_i-\epsilon]$ holds for each $i=1,2,\dots,m$ and $t\ge\tau$.
\end{rem}

The denominator of \eqref{eq:K} contains positive terms which can be interpreted as the inhibitory effect of other species, and the time-varying coefficient $k_{ij}(t)$ introduces the dependence of the system parameters on various factors such as temperature or the dynamical behaviour of other species that are not explicitly modelled as state variables. This class of rate functions contains many well-known examples, demonstrating the range and flexibility of reaction rates of the above form:
\begin{enumerate}
	\item Setting each $\theta_i(n_i)=n_i$ and $\nu_j(s_j)=s_j$ and $\Psi_{ij}(n,s)=0$ we obtain the case of classical mass-action kinetics with time-varying rate coefficients: $\mathcal{K}_{ij}(n,s,t)=k_{ij}(t)n_is_j$.
	\item Setting each $\theta_i(n_i)=n_i$ and $\nu_j(s_j)=s_j$ and $\Psi_{ij}(n,s)=l^2-1+ln_i+ls_j+n_is_j$ for some $l>0$ yields
		\begin{equation}
			\mathcal{K}_{ij}(n,s,t)=k_{ij}(t)\frac{n_is_j}{(l+n_i)(l+s_j)}
		\end{equation}
		corresponding to simple saturating kinetics described by the Monod equation.
	\item The previous example can also be obtained by setting $\theta_i(n_i)=\frac{n_i}{l+n_i}$ and $\nu_j(s_j)=\frac{s_j}{l+s_j}$ and $\Psi_{ij}(n,s)=0$, showing that \eqref{eq:K} is not unique. Notice however, that for fixed $\theta_i,\nu_j$ transformations the function $\Psi_{ij}$, and thus the fraction itself, is unique.
	\item Setting each $\theta_i(n_i)=\frac{n_i^L}{l+n_i^L}$ and $\nu_j(s_j)=\frac{s_j^L}{l+s_j^L}$ for some $l>0$ yields the classical Hill kinetics.
\end{enumerate}

\subsubsection*{Example 1.3}
Let us again consider the triangular compartmental model from Figure \ref{fig:comp}. For this example we set $c_1=c_2=c_3=100$, $l=100$ and
\begin{equation}
	\begin{aligned}
		\mathcal{K}_{12}(n_1,c_2-n_2,t)&=k_{12}(t)\frac{n_1(c_2-n_2)}{(l+n_1)(l+c_2-n_2)},\\
		\mathcal{K}_{23}(n_2,c_3-n_3,t)&=k_{23}(t)\frac{n_2(c_3-n_3)}{(l+n_2)(l+c_3-n_3)},\\
		\mathcal{K}_{31}(n_3,c_1-n_1,t)&=k_{31}(t)\frac{n_3(c_1-n_1)}{(l+n_3)(l+c_1-n_1)},
	\end{aligned}
\end{equation}
where the coefficient functions are considered to be exponentially decaying perturbations of the nominal values
\begin{equation}
	\bar k_{12}=40\qquad\bar k_{23}=25\qquad\bar k_{31}=50
\end{equation}
of the form
\begin{equation}
	k_{12}(t)=\bar k_{12}\qty\big(1+e^{-\frac{3t}{100}})\qquad k_{23}\qty\big(1+e^{-\frac{5t}{100}})\qquad k_{31}\qty\big(1+e^{-\frac{2t}{100}}).
\end{equation}
As a comparison let us consider the solution $\tilde n(t)$ of the time-invariant system with the above nominal values. Figure \ref{fig:tv:phase} shows the phase portrait of the perturbed and the original systems starting from various initial conditions with $H\qty\big(n(0))=150$. Figure \ref{fig:tv:time} shows the time evolution of the state variables with $n(0)=[5~45~100]^{\mathrm{T}}$, where the state variables of the perturbed and the time-invariant system are depicted with blue lines and red lines, respectively. We can observe that since the time dependent terms are exponentially decaying and both systems evolve on the same linear manifold, the systems tend to the same equilibrium, as expected.
\begin{figure}[H]
	\begin{subfigure}[b]{0.49\textwidth}
		\centering
		\includegraphics[width=\textwidth]{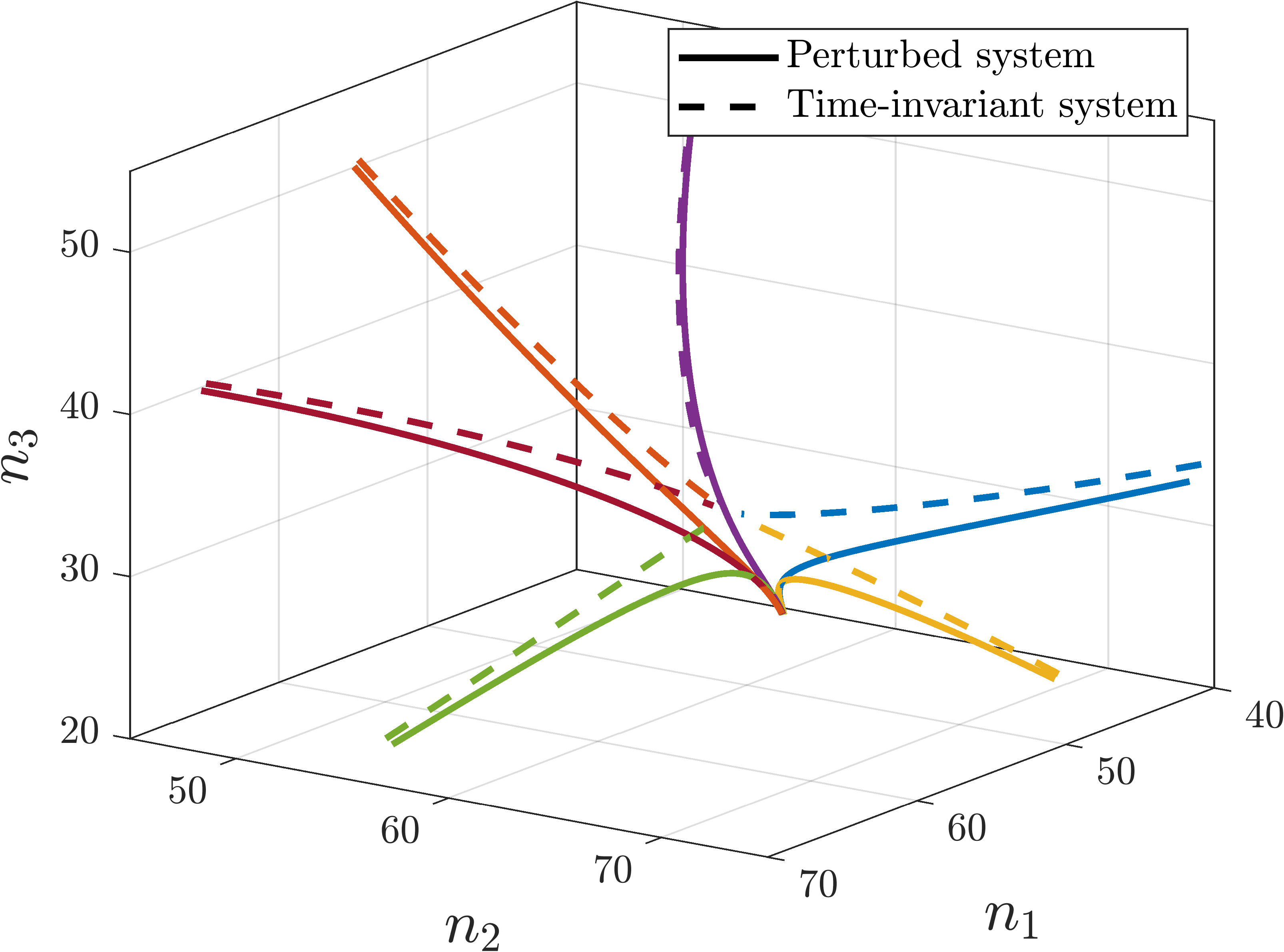}
		\caption{Phase portrait}\label{fig:tv:phase}
	\end{subfigure}
	\hfill
	\begin{subfigure}[b]{0.49\textwidth}
		\centering
		\includegraphics[width=\textwidth]{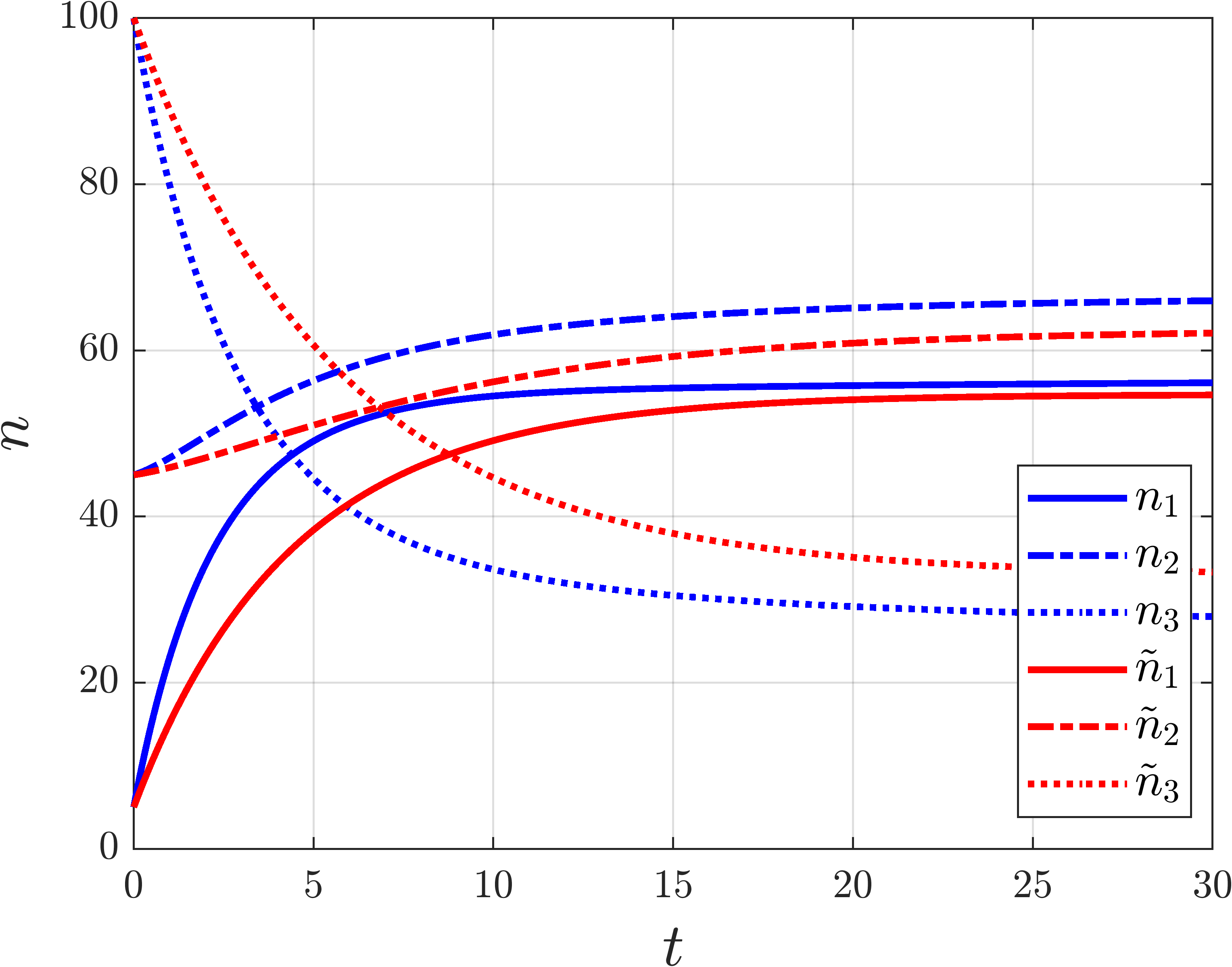}
		\caption{Time evolution of state variables}\label{fig:tv:time}
	\end{subfigure}
	\caption{Trajectories and time evolution of a time-varying model with decaying time dependence}
\end{figure}

\subsection{Periodic behaviour}
Let us consider a generalized ribosome flow in the reduced state space of the form \eqref{eq:n} with transition rates of the form \eqref{eq:K} and assume that the transition functions are $\mathcal{C}^1$ and periodic with the same period (but possibly different phase). Write \eqref{eq:n} as $\dot n=F(t,n)$ and assume that the right-hand side satisfies the following monotonicity condition: $F_i(t,x)\le F_i(t,y)$ for any two distinct points $x,y\in\tilde C$ such that $x_i=y_i$ and $x_j\le y_j$ for $j\neq i$. This condition is satisfied if, for example, the transition rates are such that $\Psi_{ij}\equiv0$; that is, if there are no inhibitory phenomena. Then the system phase locks (or entrains) with the periodic excitations.

\begin{thm}
	Consider a system of the form \eqref{eq:n} satisfying the above monotonicity assumption, where each $\mathcal{K}_{ij}(t)$ is periodic with a common period $T$. Then for each $r\in[0,c]$ there exists a unique periodic function $\phi_r:\overline{\mathbb{R}}_+:\mapsto\tilde C$ with period $T$ such that for all $a\in L_r$ we have that
	\begin{equation}
		\lim_{t\rightarrow\infty}\norm{\rho(t,a)-\phi_r(t)}_{L^1}=0.
	\end{equation}
\end{thm}
\begin{proof}
	The properties of the rate functions and the fact that $\nabla H$ is positive implies the result via \cite{Tang1993}, \cite{Ji-Fa1996}.
\end{proof}

\begin{rem}
	Since, in some sense, time-invariant systems can be seen as periodic, the stability result \cite[Proposition 5.5]{szederkenyi2022persistence} is a special case of the above theorem, where $\phi_r$ is reduced to a single point of the manifold $L_r$.
\end{rem}

\subsubsection*{Example 1.4}
Let us again consider the triangular compartmental model from Figure \ref{fig:comp}. For this example we set $c_1=c_2=c_3=100$ and
\begin{equation}
	\begin{aligned}
		\mathcal{K}_{12}(n_1,c_2-n_2,t)&=100\qty\big(3+2\cos(t+0.5))\frac{n_1(c_2-n_2)}{(l+n_1)(l+c_2-n_2)},\\
		\mathcal{K}_{23}(n_2,c_3-n_3,t)&=100\qty\big(7+5\sin(3t-2.5))\frac{n_2(c_3-n_3)}{(l+n_2)(l+c_3-n_3)},\\
		\mathcal{K}_{31}(n_3,c_1-n_1,t)&=100\qty\big(2+\cos(2t-1))\frac{n_3(c_1-n_1)}{(l+n_3)(l+c_1-n_1)},
	\end{aligned}
\end{equation}
which clearly have the same period $T=2\pi$. Figures \ref{fig:ent:phase} and \ref{fig:ent:time} show the phase portrait of the system starting from various initial conditions with $l=100$, $H\qty\big(n(0))=150$ and the time evolution of the state variables with $n(0)=[5~45~100]^{\mathrm{T}}$, respectively.
\begin{figure}[H]
	\begin{subfigure}[b]{0.49\textwidth}
		\centering
		\includegraphics[width=\textwidth]{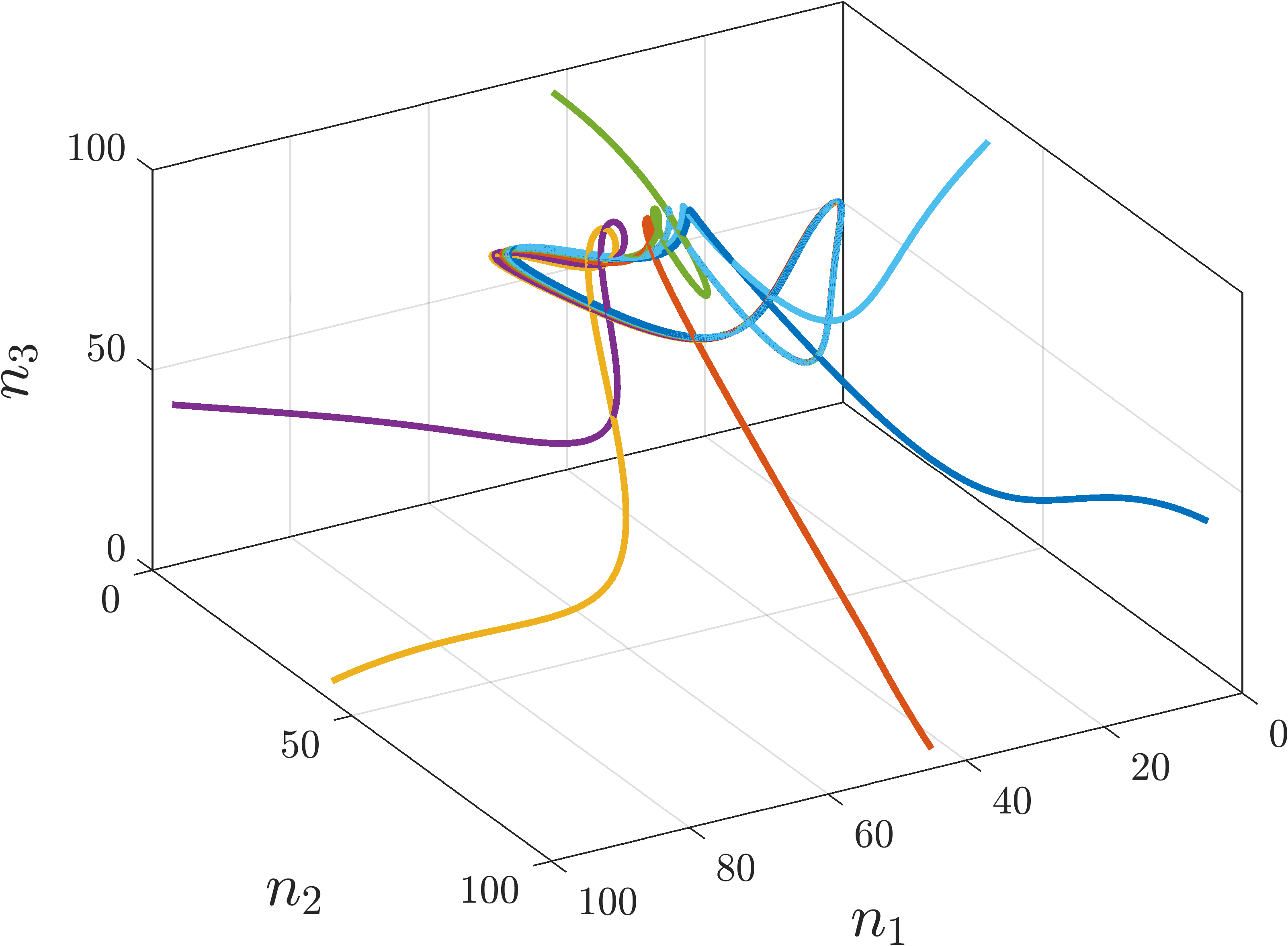}
		\caption{Phase portrait}\label{fig:ent:phase}
	\end{subfigure}
	\hfill
	\begin{subfigure}[b]{0.49\textwidth}
		\centering
		\includegraphics[width=\textwidth]{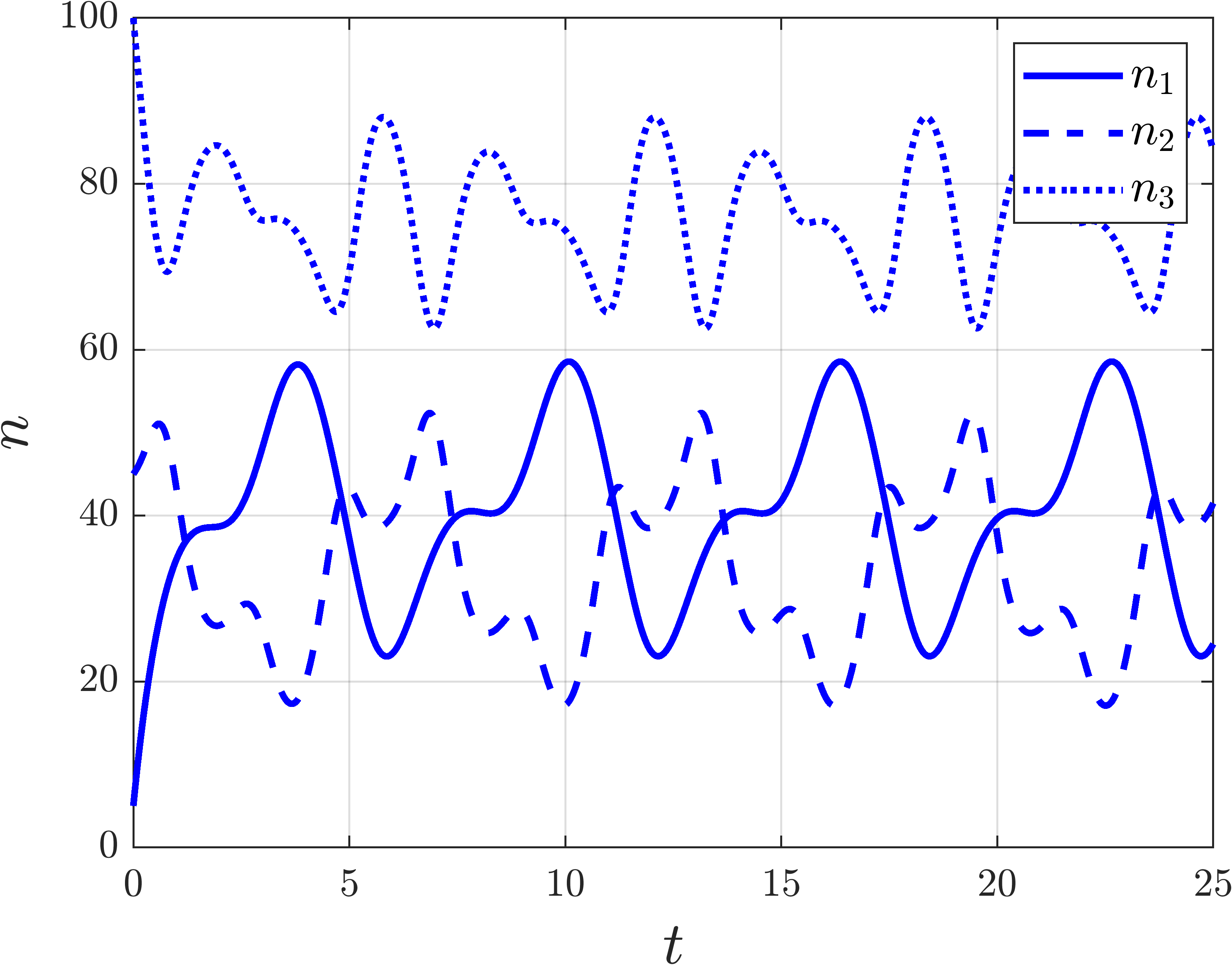}
		\caption{Time evolution of state variables}\label{fig:ent:time}
	\end{subfigure}
	\caption{Entrainment of a strongly connected model with periodic transition rates}
\end{figure}

\section{Stability analysis}\label{sec:sta}
In this section we show that generalized ribosome flows with reaction rate functions of the form \eqref{eq:K} with piecewise locally Lipschitz $k_{ij}(t)$ coefficients satisfy a certain notion of robustness to the changes in the time-varying rate functions that can be traced back to the input-to-state stability of rate-controlled biochemical networks thoroughly investigated in \cite{Chaves2005}. The main difficulty in applying these results lies in the aforementioned fact that the CRN assigned to a compartmental model is generally not weakly reversible and its deficiency is generally not zero (see, Theorem \ref{thm:deficiency}) even if the compartmental topology is strongly connected. In order to circumvent this, we will perform a model reduction and rewrite \eqref{eq:n} by factoring out certain terms. Let us first state the most important notions and results of \cite{Chaves2005}.\newline

Consider the system corresponding to a CRN with $R$ reactions
\begin{equation}\label{eq:dotxf}
	\dot x=f(x,u)=\sum_{i=1}^R\sum_{j=1}^Ru_{ij}(t)\prod_{l=1}^n\theta_i^{y_{ij}}(x_i)[y_i-y_j],
\end{equation}
where the nonnegative functions $u_{ij}$ are piecewise locally Lipschitz with a finite number of discontinuities and the stoichiometric coefficient vectors $y_i,y_j$ are as described in \ref{subsec:CRNs}. In this section, however, we restrict the conditions on the transformation functions $\theta_i:\overline{\mathbb{R}}_+\mapsto[0,\infty)$. Namely, we assume that
\begin{enumerate}[(a)]
	\item $\theta_i$ is real analytic,
	\item $\theta_i(0)=0$,
	\item $\int_0^1|\log\theta_i(r)|\dd{r}<\infty$
	\item $\theta_i$ is strictly increasing and onto the set $[0,\sigma_i)$ for some $\sigma_i\in[0,\infty)$,
	\item $\lim_{t\rightarrow\log\sigma_i}\int_a^t\rho_i^{-1}(r)\dd{r}-pt=\infty$ for any $a<\log\sigma_i$ and any constant $p>0$, where $\rho_i=\log\theta_i$.
\end{enumerate}

Before continuing with the definitions, we consider the case when $u(t)$ is a constant matrix $A$. We assume that $A$ has nonnegative entries and is irreducible; that is, the underlying reaction graph is strongly connected. We denote the set of such $A$ matrices as $\mathcal{A}$. Then the equilibria of $\dot x=f(x,A)$ can be divided into the sets of boundary equilibria and positive equilibria:
\begin{equation}
	\begin{aligned}
		&E_0=\qty\big{x\in\partial\overline{\mathbb{R}}_+^n\big|f(x,A)=0},\\
		&E_{A,+}=\qty\big{x\in\mathbb{R}_+^n\big|f(x,A)=0}.
	\end{aligned}
\end{equation}
Then, the result \cite[Theorem 2.1]{Chaves2005} (and also \cite[Theorem 2]{Son:2001}) shows that if there are no boundary equilibria in any positive class, then each positive class contains a unique globally (relative to the positive class) asymptotically stable positive equilibrium. Denote the unique positive equilibrium in the same class as $x_0$ as $\overline x(x_0,A)$ and notice that $E_{A,+}=\qty\big{\overline x(x_0,A)\big|x_0\in\mathbb{R}_+^n}$. Finally, denote
\begin{equation}
	\mathcal{E}=\bigcup_{A\in\mathcal{A}}E_{A,+}.
\end{equation}
\begin{definition} We define the following function classes:
	\begin{enumerate}[(i)]
		\item A function $\alpha:\overline{\mathbb{R}}_+\mapsto\overline{\mathbb{R}}_+$ is said to be of class $\mathcal{K}$ if it is continuous, strictly increasing and has $\alpha(0)=0$.
		\item The subset of unbounded functions of class $\mathcal{K}$ are denoted by $\mathcal{K}_{\infty}$.
		\item A function $\beta:\overline{\mathbb{R}}_+\times\overline{\mathbb{R}}_+$ is said to be of class $\mathcal{KL}$ if $\beta(.,t)$ is of class $\mathcal{K}$ for all $t\ge0$ and $\beta(r,.)$ is strictly decreasing to zero for all $r>0$.
	\end{enumerate}
\end{definition}
We consider nonnegative time-varying inputs such that at any time instant the reaction graph is strongly connected; that is, the input-value set $\mathbb{U}$ is a subset of $\mathcal{A}$. Furthermore, let $\norm{.}_2$ denote the spectral norm induced by the Euclidian norm and for $u:\overline{\mathbb{R}}_+\mapsto\mathbb{U}$ define
\begin{equation}
	\norm{u}_{\mathbb{U}}=\essup_{t\in[0,\infty)}\norm{u(t)}_2.
\end{equation}
\begin{definition}
	A system $\dot x=f(x,u)$ is uniformly input-to-state stable (ISS) with input-value set $\mathbb{U}$ if for every compact set $P\subset\mathcal{E}$ and every compact set $F\subset\overline{\mathbb{R}}_+^n$ containing $P$, there exist functions $\beta=\beta_P$ of class $\mathcal{KL}$ and $\phi=\phi_P$ of class $\mathcal{K}_{\infty}$ such that, for every $\overline x_o\in P\cap E_{u_0,+}$ for some $u_0\in\mathbb{U}$ we have that
	\begin{equation}
		\norm{x(t)-\overline x_0}\le\beta\qty\big(\norm{x_0-\overline x_0},t)+\phi\qty\big(\norm{u-u_0}_{\mathbb{U}})
	\end{equation}
	holds for each $u:\overline{\mathbb{R}}_+\mapsto\mathbb{U}$ input and every initial condition $x_0\in F\cap\mathcal{S}_{\overline x_0}$ and for all $t\ge 0$ such that $x(s)\in F$ for $s\in[0,t]$.
\end{definition}
According to the above definition we say that a system is ISS if it is globally asymptotically system in the absence of external inputs and if its trajectories are bounded by an appropriate function of the input. In some sense this definition is intended to capture the idea of "bounded input bounded output" stability, since for bounded $u$ input ($u-u_0$ to be more precise) the trajectories will remain in a ball and, in fact, approach the ball $\phi(\norm{u-u_0}_{\mathbb{U}})$ as $t$ increases \cite{Sontag1989}.\newline

We assume that there exists a uniform lower bound on the parameters; that is, we consider input-value sets of the form
\begin{equation}
	\mathcal{A}\supset\mathbb{U}_{\epsilon}=\qty\big{u\in\mathcal{A}\big|u_{ij}(t)\ge\epsilon~\forall t\ge0\text{, or }u_{ij}(t)=0~\forall t\ge0}.
\end{equation}
We also recall that the input functions are piecewise locally Lipschitz in time with a finite number of discontinuities, thus we introduce
\begin{equation}
	\mathcal{W}=\qty\big{w:\overline{\mathbb{R}}_+\mapsto\mathbb{U}_{\epsilon}\big|w\text{ is piecewise locally Lipschitz}}.
\end{equation}
Then the main Theorem of \cite{Chaves2005} states:
\begin{thm}\label{thm:ISS}
	Consider the system \eqref{eq:dotxf} with and suppose that is is mass-conservative; that is, there exists $v\in\mathbb{R}_+^n$ such that $v^{\mathrm{T}}f(x,u)=0$ for all $x\in\overline{\mathbb{R}}_+^n$ and $u\in\mathcal{A}$. Then the system with input maps $u\in\mathcal{W}$ is uniformly ISS with input-value set $\mathbb{U}_{\epsilon}$.
\end{thm}
The proof relies on the candidate ISS-Lyapunov function (for the definition of which and for the exact connection with ISS stability we refer to \cite{Chaves2005})
\begin{equation}\label{eq:lyap}
	V(x,\overline x)=\sum_{i=1}^n\int_{\overline x_i}^{x_i}\qty\big(\log\theta_i(r)-\log\theta_i(\overline x_i))\dd{r}
\end{equation}
which, for mass-action systems, yields the classical entropy-like Lyapunov function well-known from the theory of chemical reaction networks, see \eqref{eq:LTV_lyap}. We note that $V(x,\overline x)$ is uniquely determined by the $\theta_i$ functions and does not depend explicitly on the reaction/compartmental structure or the time-varying $u_{ij}(t)$ functions; that is, it is universal in the sense of \cite{Gorban2019}.
\begin{rem}
	We note that the assumption that the compartmental graph (and thus the reaction graph of the factored model) is strongly connected is purely technical. For time-invariant systems it simply ensures that the unique equilibrium on each level set of the first integral is positive (except for the trivial case of an empty network of course). In fact, in some cases the initial values of the network can ensure the positivity of the equilibrium even for not strongly connected systems (see Example \hyperref[ex:NSR]{2} and \cite{Vaghy2022} for more details), in which case the above Lyapunov function can be applied.
\end{rem}

\subsection{Factorization approach}\label{sec:factor}
Let us consider a generalized ribosome flow in the reduced state space of the form \eqref{eq:n}, in this case given by
\begin{equation}\label{eq:ni_start}
	\begin{aligned}
		\dot n_i&=\sum_{j\in\mathcal{D}_i}\mathcal{K}_{ji}(n,c-n,t)-\sum_{j\in\mathcal{R}_i}\mathcal{K}_{ij}(n,c-n,t)\\
		&=\sum_{j\in\mathcal{D}_i}k_{ji}(t)\frac{\theta_j(n_j)\nu_i(c_i-n_i)}{1+\Psi_{ji}(n,c^{(m)}-n)}-\sum_{j\in\mathcal{R}_i}k_{ij}(t)\frac{\theta_i(n_i)\nu_j(c_j-n_j)}{1+\Psi_{ij}(n,c^{(m)}-n)}.
	\end{aligned}
\end{equation}

Notice that we can naturally factor some terms of the transition rates into the time-varying coefficient as
\begin{equation}
	k_{ij}(t)\frac{\theta_i(n_i)\nu_j(c_j-n_j)}{1+\Psi_{ij}(n,c^{(m)}-n)}=\frac{k_{ij}(t)\nu_j(c_j-n_j)}{1+\Psi_{ij}(n,c^{(m)}-c)}\theta_i(n_i)=:\tilde k_{ij}(t)\theta_i(n_i).
\end{equation}
Then \eqref{eq:ni_start} can be rewritten as
\begin{equation}\label{eq:ni_factored}
	\dot n_i=\sum_{j\in\mathcal{D}_i}\tilde k_{ji}(t)\theta_j(n_j)-\sum_{j\in\mathcal{R}_i}\tilde k_{ij}(t)\theta_i(n_i).
\end{equation}

This equation can be clearly embedded into the class of strongly connected systems of the form \eqref{eq:dotxf}, since the reaction graph of \eqref{eq:ni_factored} consists of species $\Sigma=\qty{N_1,N_2,\dots,N_m}$, has the $m\times m$ identity matrix as its stoichiometric matrix and for each transition $(q_i,q_j)\in A$ we assign a reaction of the form
\begin{equation}
	N_i\xrightarrow{\tilde{\mathcal{K}}_{ij}(t)}N_j,
\end{equation}
and thus the system of differential equations can be written as
\begin{equation}\label{eq:n_factored}
	\dot n=I\tilde A_k(t)\theta(n)
\end{equation}
where the elements of $\tilde A_k$ are given by
\begin{equation}
	\qty\big[\tilde A_k(t)]_{ij}=\begin{cases}-\sum_{l\in\mathcal{R}_i}\tilde k_{il}(t)\quad&\text{if }i=j,\\\tilde k_{ji}(t)\quad&\text{if }j\in\mathcal{D}_i,\\0\quad&\text{otherwise}.\end{cases}
\end{equation}

Note that the fractions $\frac{\nu_j(c_j-n_j)}{1+\Psi_{ij}(n,c^{(m)}-n)}$ are differentiable (and thus Lipschitz) and each $k_{ij}(t)$ is piecewise locally Lipschitz, hence each $\tilde k_{ij}(t)$ is piecewise locally Lipschitz. This shows that generalized ribosome flows can be embedded into the class of rate-controlled biochemical networks of \cite{Chaves2005} in a way that preserves the compartmental structure; that is, the reaction graph of \eqref{eq:n_factored} is topologically isomorph to the compartmental graph. In particular if the compartmental model is strongly connected, then the reaction graph of the reduced system \eqref{eq:n_factored} is strongly connected as well. Furthermore, combining the persistence of the system with Remark \ref{rem:eps} we find that $\tilde A_k\in\mathcal{W}$, and thus Theorem \ref{thm:ISS} ensures input-to-state stability.

\subsection{Quasi-LTV factorization}\label{sec:LTV}
A classical argument shows that the model reduction above can result in a Linear Time-Varying (LTV) system \cite{Jacquez1993}. Consider an $F(x)\in\mathcal{C}^k(\mathbb{R})$ nonnegative function such that $F(0)=0$, where $k\ge1$. Then for the function $F(rx)$ we have
\begin{equation}
	\dv{F(rx)}{r}=xF'(rx)
\end{equation}
and thus
\begin{equation}
	F(x)-F(0)=x\int_0^1F'(rx)\dd{r}=xf(x)
\end{equation}
and since $F(0)=0$, we find that $F(x)=xf(x)$. Note, that the calculation also shows that $f\in\mathcal{C}^{k-1}(\mathbb{R})$. Since $\theta_i$ is real analytic we have that $\theta_i(n_i)=\hat\theta_i(n_i)n_i$ for some $\hat\theta_i$ real analytic function. Then \eqref{eq:ni_factored} can be rewritten as
\begin{equation}\label{eq:ni_linear}
	\dot n_i=\sum_{j\in\mathcal{D}_i}\hat k_{ji}(t)n_j-\sum_{j\in\mathcal{R}_i}\hat k_{ij}(t)n_i
\end{equation}
where
\begin{equation}
	\hat k_{ij}(t)=\frac{k_{ij}(t)\hat\theta_i(n_i)\nu_j(c_j-n_j)}{1+\Psi_{ij}(n,c^{(m)}-n)}.
\end{equation}
Similarly as before, the reaction graph of \eqref{eq:ni_linear} consists of species $\Sigma=\qty{N_1,N_2,\dots,N_m}$, has the $m\times m$ identity matrix as its stoichiometric matrix and for each transition $(q_i,q_j)\in A$ we assign a reaction of the form
\begin{equation}
	N_i\xrightarrow{\hat{\mathcal{K}}_{ij}(t)}N_j,
\end{equation}
and thus the system of differential equations can be written as
\begin{equation}\label{eq:n_linear}
	\dot n=I\hat A_k(t)n
\end{equation}
where the elements of $\hat A_k$ are given by
\begin{equation}
	\qty\big[\hat A_k(t)]_{ij}=\begin{cases}-\sum_{l\in\mathcal{R}_i}\hat k_{il}(t)\quad&\text{if }i=j,\\\hat k_{ji}(t)\quad&\text{if }j\in\mathcal{D}_i,\\0\quad&\text{otherwise}.\end{cases}
\end{equation}
Again, each $\hat k_{ij}(t)$ is piecewise locally Lipschitz, thus for strongly connected compartmental models Theorem \ref{thm:ISS} ensures input-to-state stability via Remark \ref{rem:eps}.

\subsection{Factorization of Monod kinetics}
Let us consider the triangular model in Figure \ref{fig:comp} with rational kinetics corresponding to Monod kinetics of the form
\begin{equation}\label{eq:ex:orig}
	\begin{aligned}
		\dot n_1&=k_{31}(t)\frac{n_3}{l+n_3}\frac{c_1-n_1}{l+c_1-n_1}-k_{12}(t)\frac{n_1}{l+n_1}\frac{c_2-n_2}{l+c_2-n_2}\\
		\dot n_2&=k_{12}(t)\frac{n_1}{l+n_1}\frac{c_2-n_2}{l+c_2-n_2}-k_{23}(t)\frac{n_2}{l+n_2}\frac{c_3-n_3}{l+c_3-n_3}\\
		\dot n_3&=k_{23}(t)\frac{n_2}{l+n_2}\frac{c_3-n_3}{l+c_3-n_3}-k_{31}(t)\frac{n_3}{l+n_3}\frac{c_1-n_1}{l+c_1-n_1}
	\end{aligned}
\end{equation}
for some $l>0$. As discussed before, the corresponding CRN is not strongly connected. However, using the functions
\begin{equation}
		\tilde k_{31}(t)=k_{31}(t)\frac{c_1-n_1}{l+c_1-n_1}\qquad\tilde k_{12}(t)=k_{12}(t)\frac{c_2-n_2}{l+c_2-n_2}\qquad\tilde k_{23}(t)=k_{23}(t)\frac{c_3-n_3}{l+c_3-n_3}
\end{equation}
we can to rewrite \eqref{eq:ex:orig} as
\begin{equation}\label{eq:ex:fact1}
	\begin{aligned}
		\dot n_1&=\tilde k_{31}(t)\frac{n_3}{l+n_3}-\tilde k_{12}(t)\frac{n_1}{l+n_1}\\
		\dot n_2&=\tilde k_{12}(t)\frac{n_1}{l+n_1}-\tilde k_{23}(t)\frac{n_2}{l+n_2}\\
		\dot n_3&=\tilde k_{23}(t)\frac{n_2}{l+n_2}-\tilde k_{31}(t)\frac{n_3}{l+n_3}.
	\end{aligned}
\end{equation}
Then the CRN corresponding to \eqref{eq:ex:fact1} has the following species and reactions:
\begin{equation}
	\begin{aligned}
		&\Sigma=\qty{N_1,N_2,N_3}\\
		&R_1:N_1\xrightarrow{\tilde k_{12}}N_2\\
		&R_2:N_2\xrightarrow{\tilde k_{23}}N_3\\
		&R_3:N_3\xrightarrow{\tilde k_{31}}N_1.
	\end{aligned}
\end{equation}
which is strongly connected and isomorph to the compartmental model in Figure \ref{fig:comp}. We arrive at the same conclusion if we instead use the functions
\begin{equation}
	\begin{aligned}
		\hat k_{31}(t)&=k_{31}(t)\frac{1}{l+n_3}\frac{c_1-n_1}{l+c_1-n_1}\\
		\hat k_{12}(t)&=k_{12}(t)\frac{1}{l+n_1}\frac{c_2-n_2}{l+c_2-n_2}\\
		\hat k_{23}(t)&=k_{23}(t)\frac{1}{l+n_2}\frac{c_3-n_3}{l+c_3-n_3}
	\end{aligned}
\end{equation}
to rewrite \eqref{eq:ex:orig} as
\begin{equation}\label{eq:ex:fact1}
	\begin{aligned}
		\dot n_1&=\hat k_{31}(t)n_3-\hat k_{12}(t)n_1\\
		\dot n_2&=\hat k_{12}(t)n_1-\hat k_{23}(t)n_2\\
		\dot n_3&=\hat k_{23}(t)n_2-\hat k_{31}(t)n_3.
	\end{aligned}
\end{equation}
Note that the quasi-LTV factorization might be more complicated in some cases, but the investigation in Section \ref{sec:LTV} guarantees its existence. 

\subsection{Induced family of Lyapunov functions}
The above investigation demonstrates that generalized ribosome flows can be embedded into rate-controlled biochemical networks in at least two different ways, where each embedding induces a different Lyapunov function of the form \eqref{eq:lyap}. Thus, in general, we may use at least two different Lyapunov functions governing the same dynamics. To characterize their exact connection, consider a factored system of the form \eqref{eq:n_factored} with its ISS-Lyapunov function $V(n,\overline n)$. The quasi-LTV representation of the system admits an ISS-Lyapunov function of the form
\begin{equation}\label{eq:LTV_lyap}
	V^{LTV}(n,\overline n)=\sum_{i=1}^m\int_{\overline n_i}^{n_i}\qty\big(\log r-\log\overline n_i)\dd{r}=\sum_{i=1}^m\qty\bigg(n_i\log\frac{n_i}{\overline n_i}+\overline n_i-n_i)=:\sum_{i=1}^mV_i^{LTV}(n_i,\overline n_i)
\end{equation}
so that we can write
\begin{equation}
	\begin{aligned}
		V(n,\overline n)&=\sum_{i=1}^m\int_{\overline n_i}^{n_i}\qty\Big(\log\qty\big(\hat\theta_i(r)r)-\log\qty\big(\hat\theta(\overline n_i)\overline n_i))\dd{r}=\sum_{i=1}^m\int_{\overline n_i}^{n_i}\qty\big(\log\hat\theta_i(r)-\log\hat\theta_i(\overline n_i))\dd{r}\\
		&+\sum_{i=1}^m\int_{\overline n_i}^{n_i}\qty\big(\log r-\log\overline n_i)\dd{r}=\sum_{i=1}^m\int_{\overline n_i}^{n_i}\qty\big(\log\hat\theta_i(r)-\log\hat\theta_i(\overline n_i))\dd{r}+V^{LTV}(n,\overline n).
	\end{aligned}
\end{equation}
\smallskip

While in general we are restricted to the above factorizations, in some special cases we may use a whole family of factorizations and corresponding Lyapunov functions. To illustrate this, consider an example when each $\theta_i(r)=\frac{r^{a_i}}{(l+r)^{b_i}}$ for some $l>0$ and $a_i\in\mathbb{N}$, $b_i\in\mathbb{N}_0$, $a_i\ge b_i$ (these properties ensure that the functions $\theta_i$ are nondecreasing). Then, after the factorization described in Section \ref{sec:factor}, the Lyapunov function \eqref{eq:lyap} becomes
\begin{equation}\label{eq:lyap_kab}
	V^{(l,a,b)}(n,\overline n)=\sum_{i=1}^m\qty\bigg((a_i-b_i)(\overline n_i-n_i)+a_in_i\log\frac{n_i}{\overline n_i}+b_i(l+n_i)\log\frac{l+\overline n_i}{l+n_i}).
\end{equation}
We emphasize that \eqref{eq:lyap} only depends on the $\theta_i$ functions, in this case parametrized with the $l,a_i,b_i$ values; that is, it is independent of the network structure and transition rate coefficients. We can also perform the factorization $\theta_i(r)=\tilde\theta_i(r)\frac{r^{\hat a_i}}{(l+r)^{\hat b_i}}$ with $\hat a_i\in\mathbb{N}$, $\hat a_i<a_i$, $\hat b_i\in\mathbb{N}_0$, $\hat a_i\ge\hat b_i$ yielding the Lyapunov function $V^{(l,\hat a,\hat b)}$ of the same form as in \eqref{eq:lyap_kab}. This shows that the parameters $a$ and $b$ can be freely (apart from the constraints above) chosen in \eqref{eq:lyap_kab}. We may also observe some interesting behaviour at the extrema of the parameters $\hat b$ and $l$, namely, that if we choose each $\hat b_i=0$ then the Lyapunov function in \eqref{eq:lyap_kab} is independent of $l$; that is, we have that
\begin{equation}
	V^{(l,\hat a,0)}(n,\overline n)=\sum_{i=1}^m\hat a_iV_i^{LTV}(n_i,\overline n_i).
\end{equation}
Moreover, letting $l\rightarrow\infty$ yields the convergence
\begin{equation}\label{eq:Vconv}
	\lim_{l\rightarrow\infty}V^{(l,\hat a,\hat b)}(n,\overline n)=\sum_{i=1}^m\hat a_iV_i^{LTV}(n_i,\overline n_i)
\end{equation}
where $V_i^{LTV}$ is defined in \eqref{eq:LTV_lyap}.

\subsubsection*{Example 1.5}
Let us again consider a time-invariant version triangular compartmental model in the reduced state space from Figure \ref{fig:comp}. For a given initial condition $n_0$ we can substitute $n_3=H(n_0)-n_1-n_2$, and thus the Lyapunov function restricted to the manifold $\qty\big{H(n)=H(n_0)}$ can be seen as a two dimensional function with local coordinates $n_1$ and $n_2$.

We set the capacities as $c_1=c_2=c_3=100$ and $k_{12}=100$, $k_{23}=60$, $k_{31}=20$. The system has transition rates as described above with each $a_i=b_i=3$; that is, we have that
\begin{equation}
	\begin{aligned}
		\mathcal{K}_{12}(n_1,c_2-n_2)&=100\cdot\frac{n_1^3}{(l+n_1)^3}\cdot\frac{(c_2-n_2)^3}{(l+c_2-n_2)^3}\\
		\mathcal{K}_{23}(n_2,c_3-n_3)&=60\cdot\frac{n_2^3}{(l+n_2)^3}\cdot\frac{(c_3-n_3)^3}{(l+c_3-n_3)^3}\\
		\mathcal{K}_{31}(n_3,c_1-n_1)&=20\cdot\frac{n_3^3}{(l+n_3)^3}\cdot\frac{(c_1-n_1)^3}{(l+c_1-n_1)^3}.
	\end{aligned}
\end{equation}

\begin{figure}[H]
	\begin{center}
		\begin{subfigure}[b]{0.32\textwidth}
			\centering
			\includegraphics[width=\textwidth]{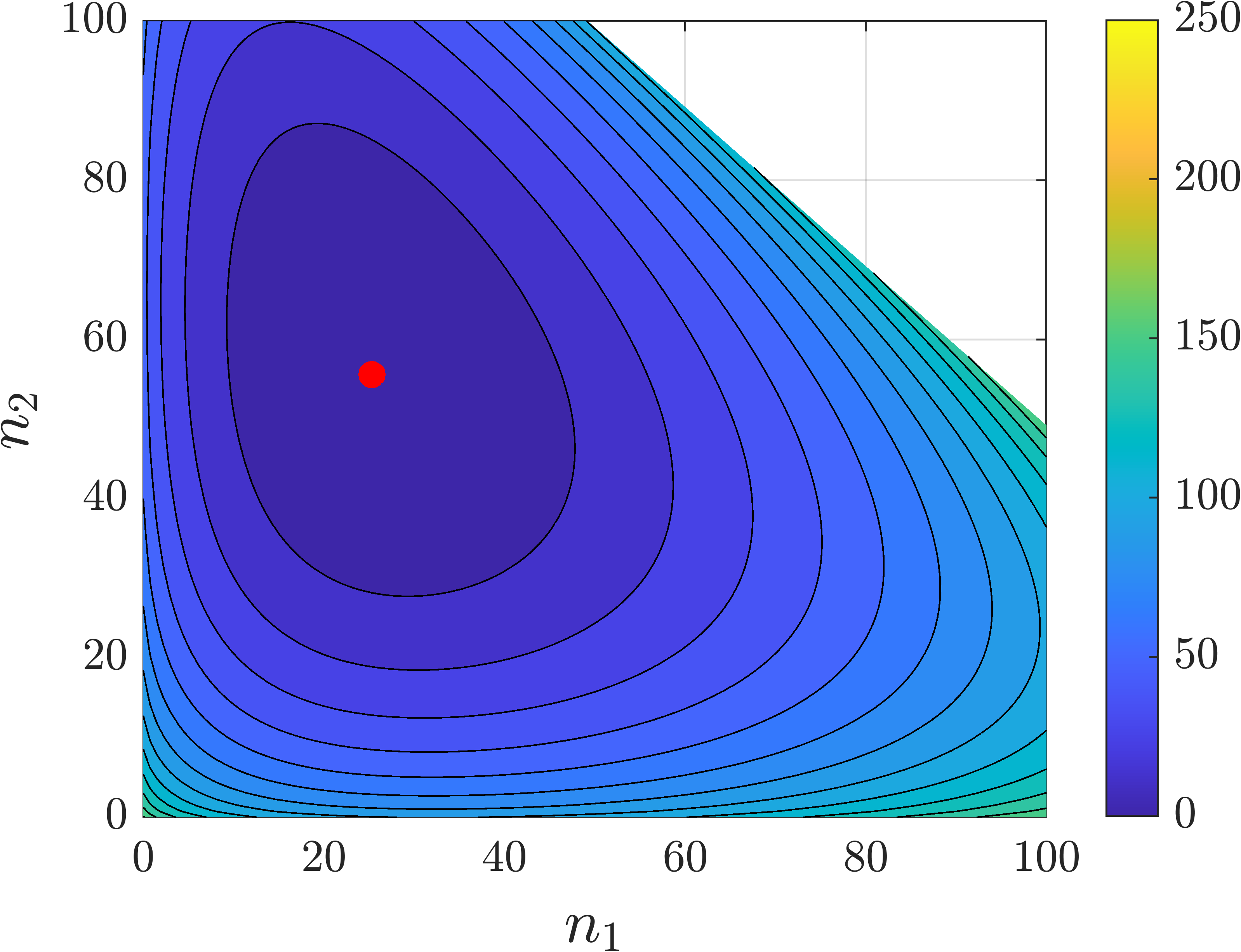}
			\caption{\scriptsize$l=25$, $\hat a=[3~3~3]$, $\hat b=[3~3~3]$}\label{fig:lyap:a}
		\end{subfigure}
		\hfill
		\begin{subfigure}[b]{0.32\textwidth}
			\centering
			\includegraphics[width=\textwidth]{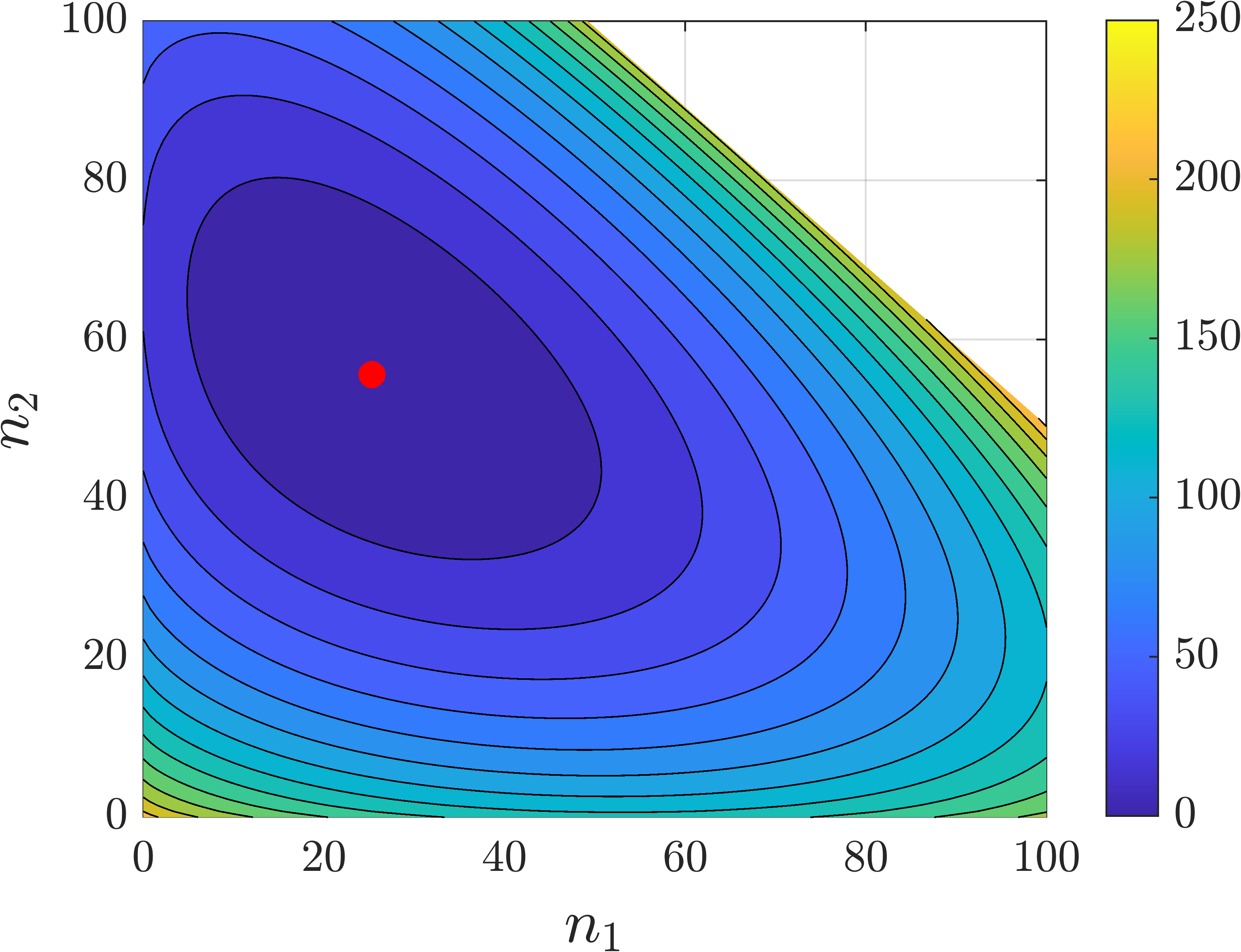}
			\caption{\scriptsize$l=25$, $\hat a=[1~2~3]$, $\hat b=[0~0~1]$}
		\end{subfigure}
		\hfill
		\begin{subfigure}[b]{0.32\textwidth}
			\centering
			\includegraphics[width=\textwidth]{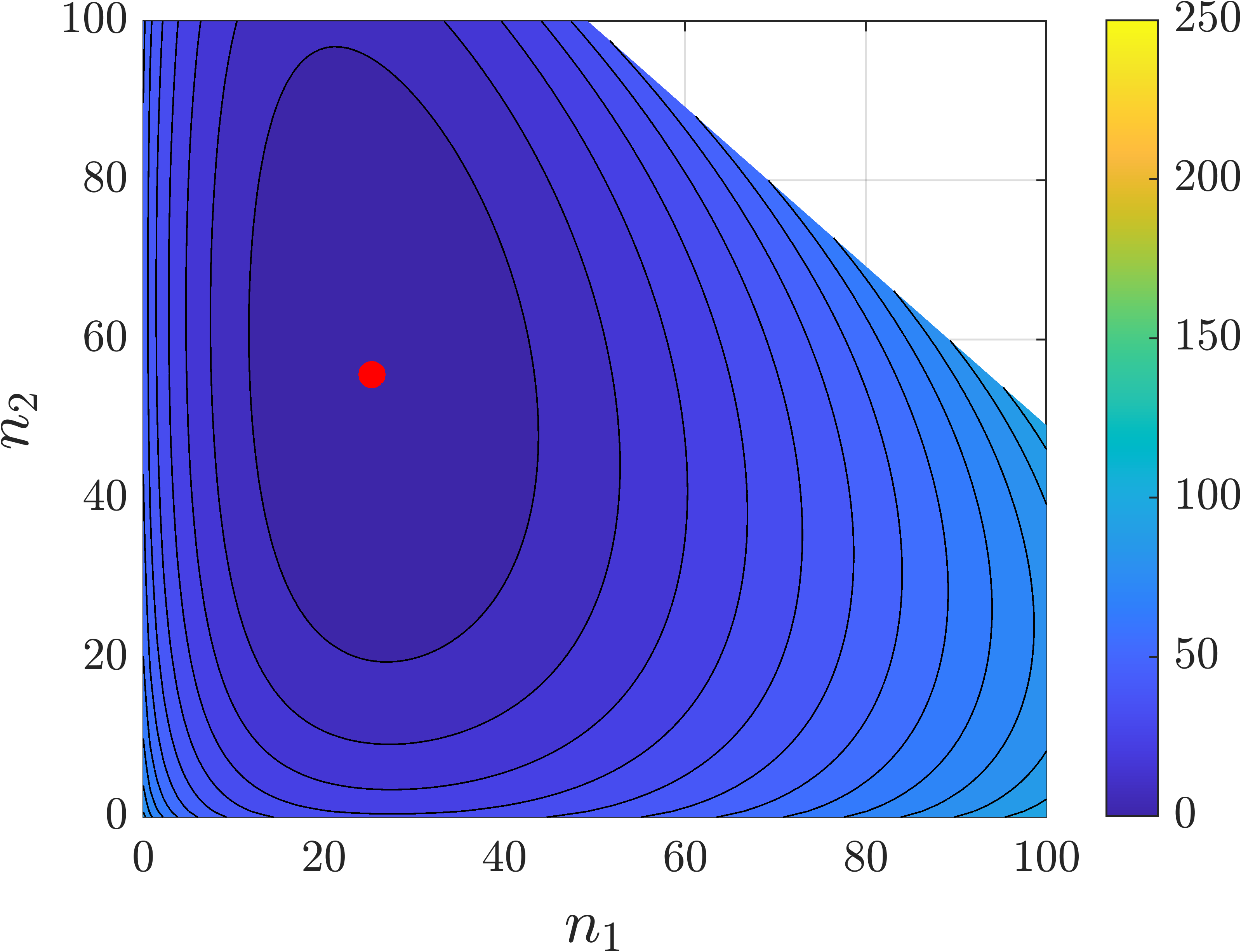}
			\caption{\scriptsize$l=25$, $\hat a=[3~1~1]$, $\hat b=[3~1~1]$}\label{fig:lyap:c}
		\end{subfigure}

		\begin{subfigure}[b]{0.32\textwidth}
			\centering
			\includegraphics[width=\textwidth]{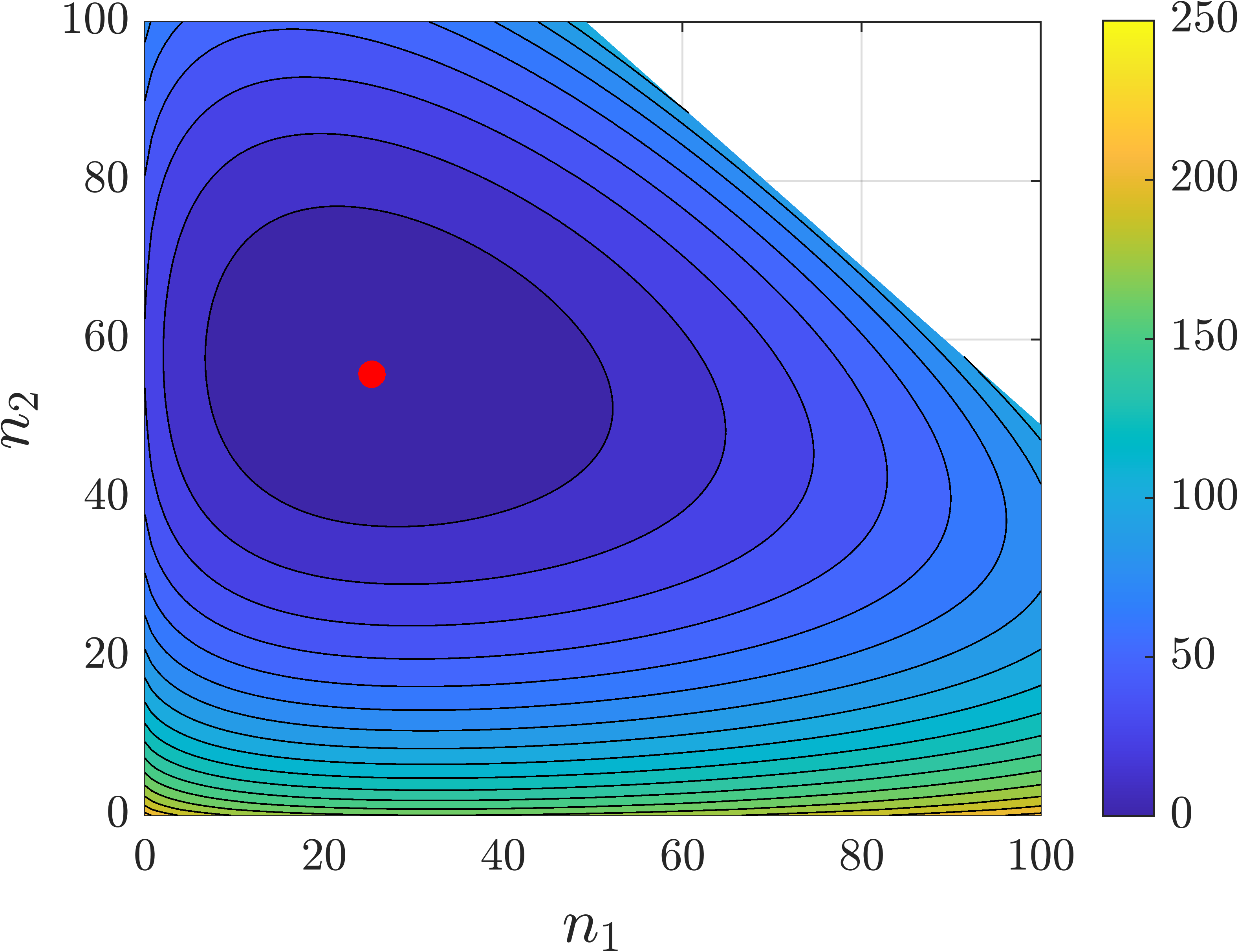}
			\caption{\scriptsize$l=25$, $\hat a=[2~3~2]$, $\hat b=[2~0~2]$}\label{fig:lyap:d}
		\end{subfigure}
		\hfill
		\begin{subfigure}[b]{0.32\textwidth}
			\centering
			\includegraphics[width=\textwidth]{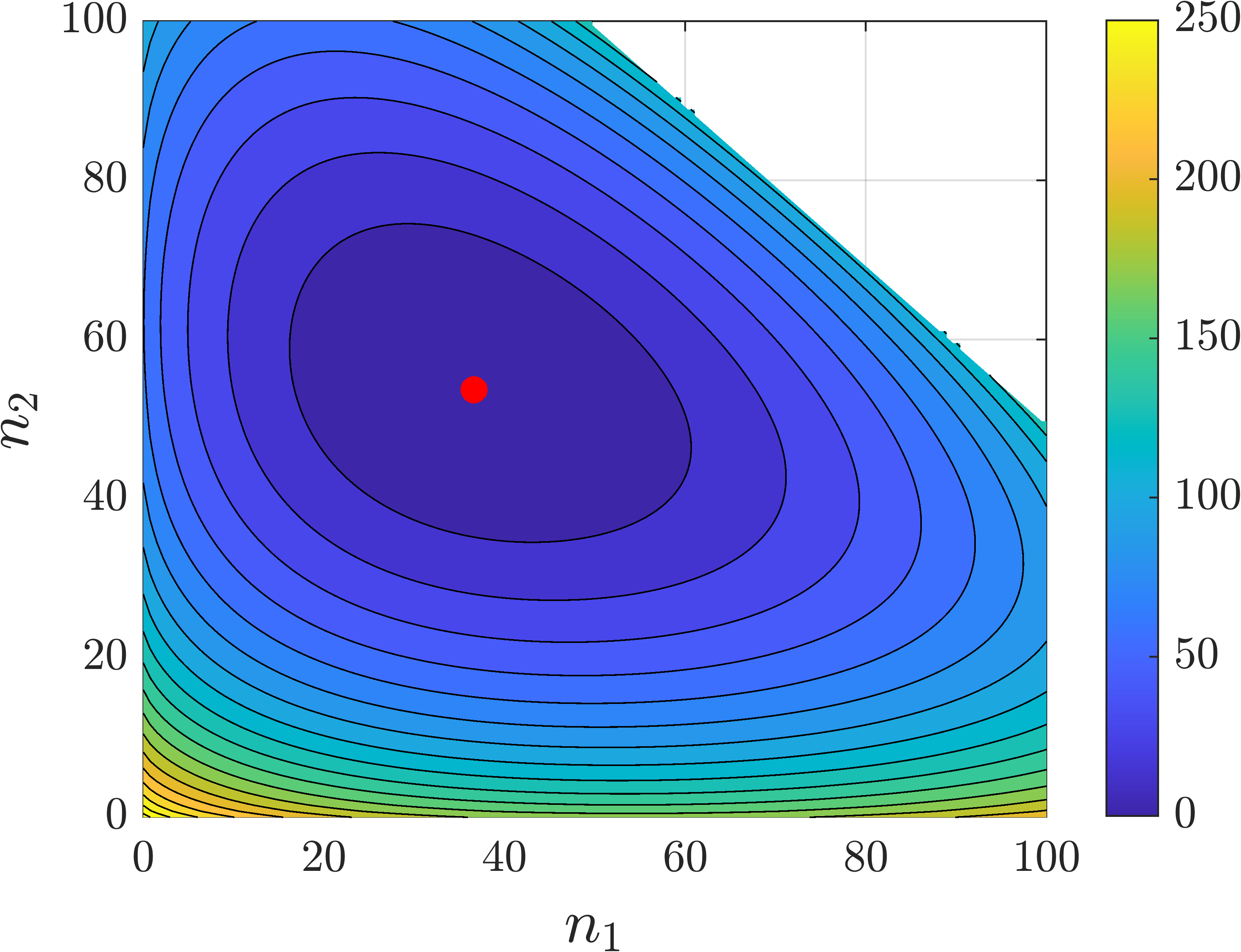}
			\caption{\scriptsize$l=100$, $\hat a=[2~3~2]$, $\hat b=[2~0~2]$}
		\end{subfigure}
		\hfill
		\begin{subfigure}[b]{0.32\textwidth}
			\centering
			\includegraphics[width=\textwidth]{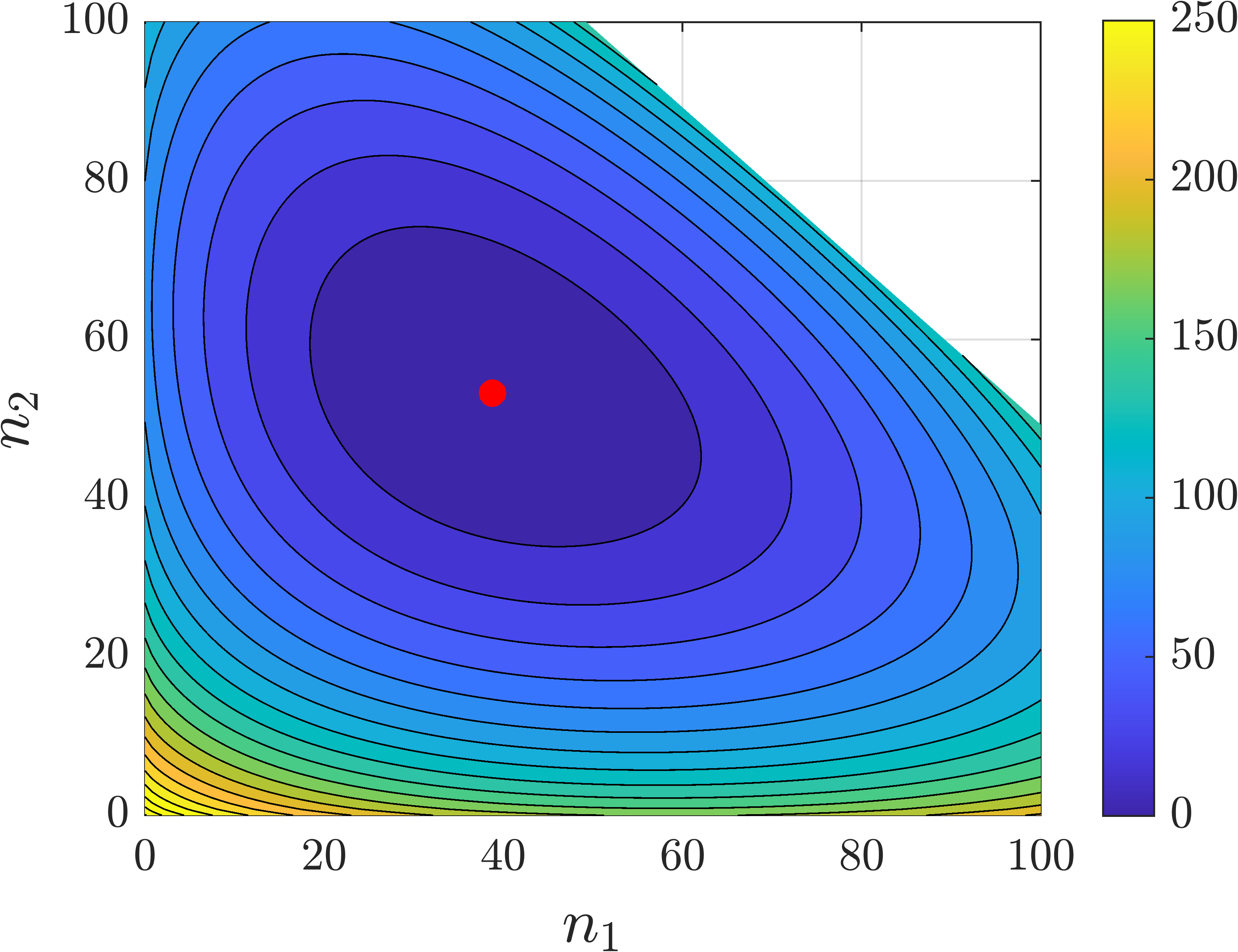}
			\caption{\scriptsize$l=200$, $\hat a=[2~3~2]$, $\hat b=[2~0~2]$}\label{fig:lyap:f}
		\end{subfigure}

		\begin{subfigure}[b]{0.32\textwidth}
			\centering
			\includegraphics[width=\textwidth]{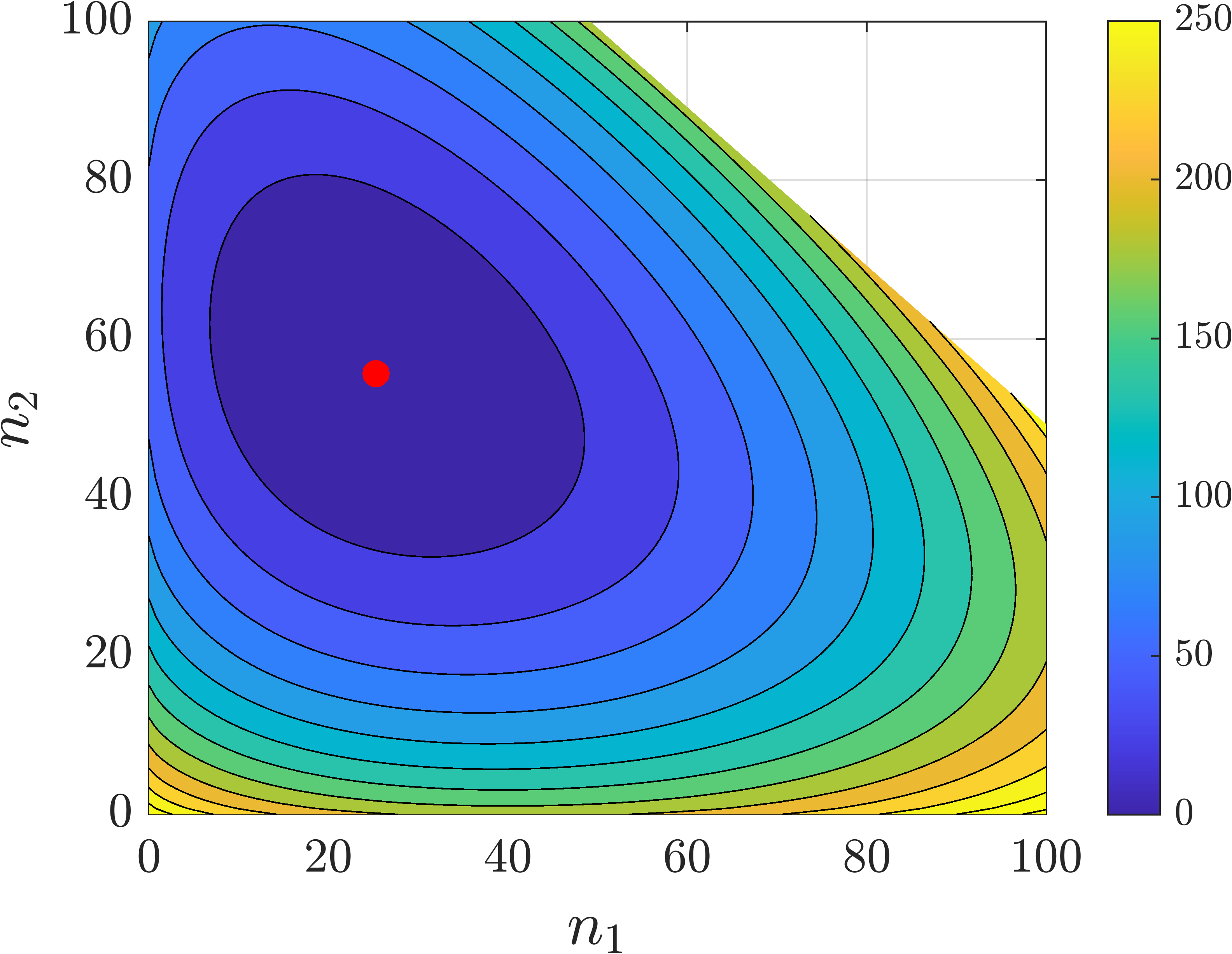}
			\caption{\scriptsize$l=25$, $\hat a=[2~3~2]$, $\hat b=[2~0~2]$}\label{fig:lyap:g}
		\end{subfigure}
		\hfill
		\begin{subfigure}[b]{0.32\textwidth}
			\centering
			\includegraphics[width=\textwidth]{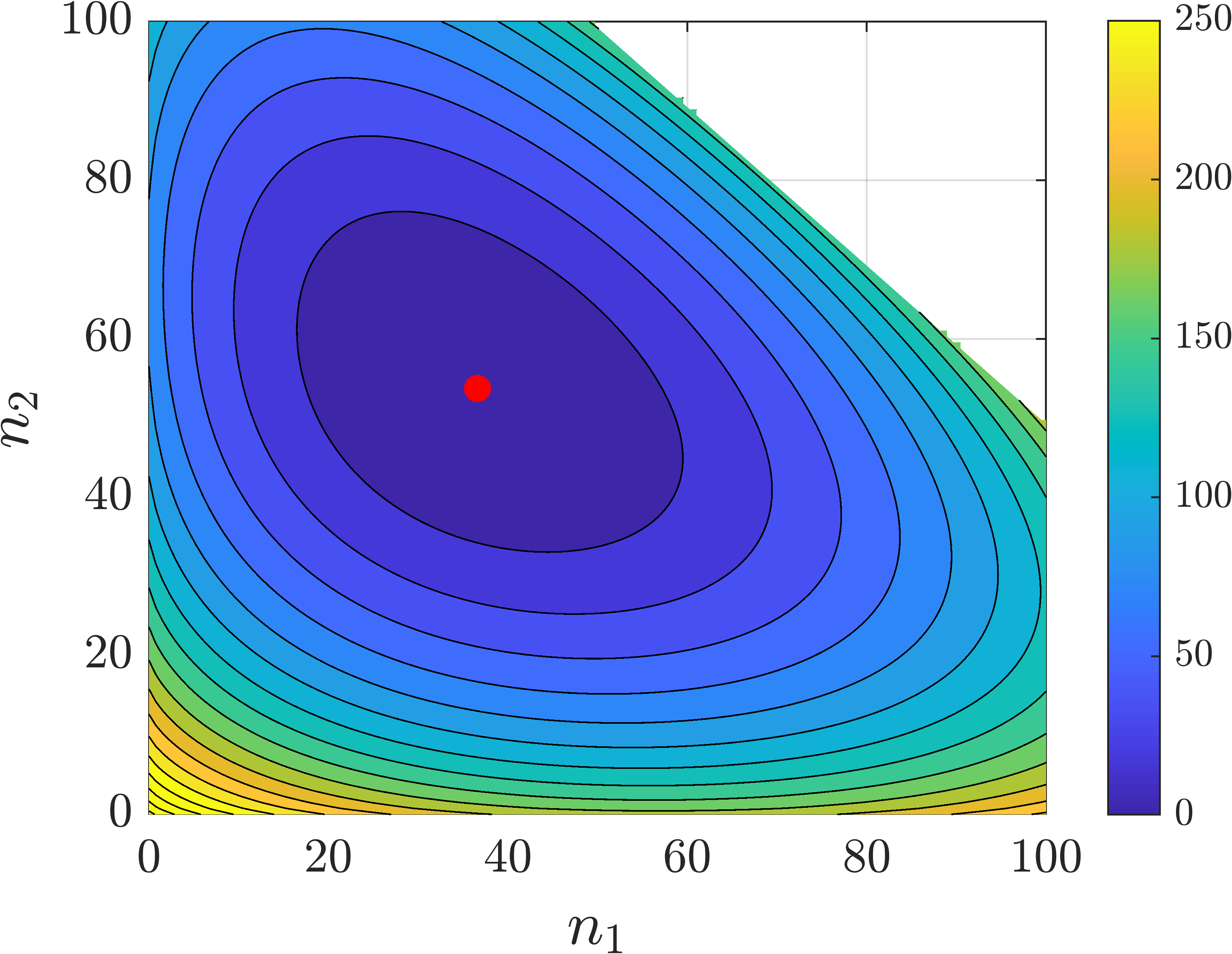}
			\caption{\scriptsize$l=100$, $\hat a=[2~3~2]$, $\hat b=[2~0~2]$}
		\end{subfigure}
		\hfill
		\begin{subfigure}[b]{0.32\textwidth}
			\centering
			\includegraphics[width=\textwidth]{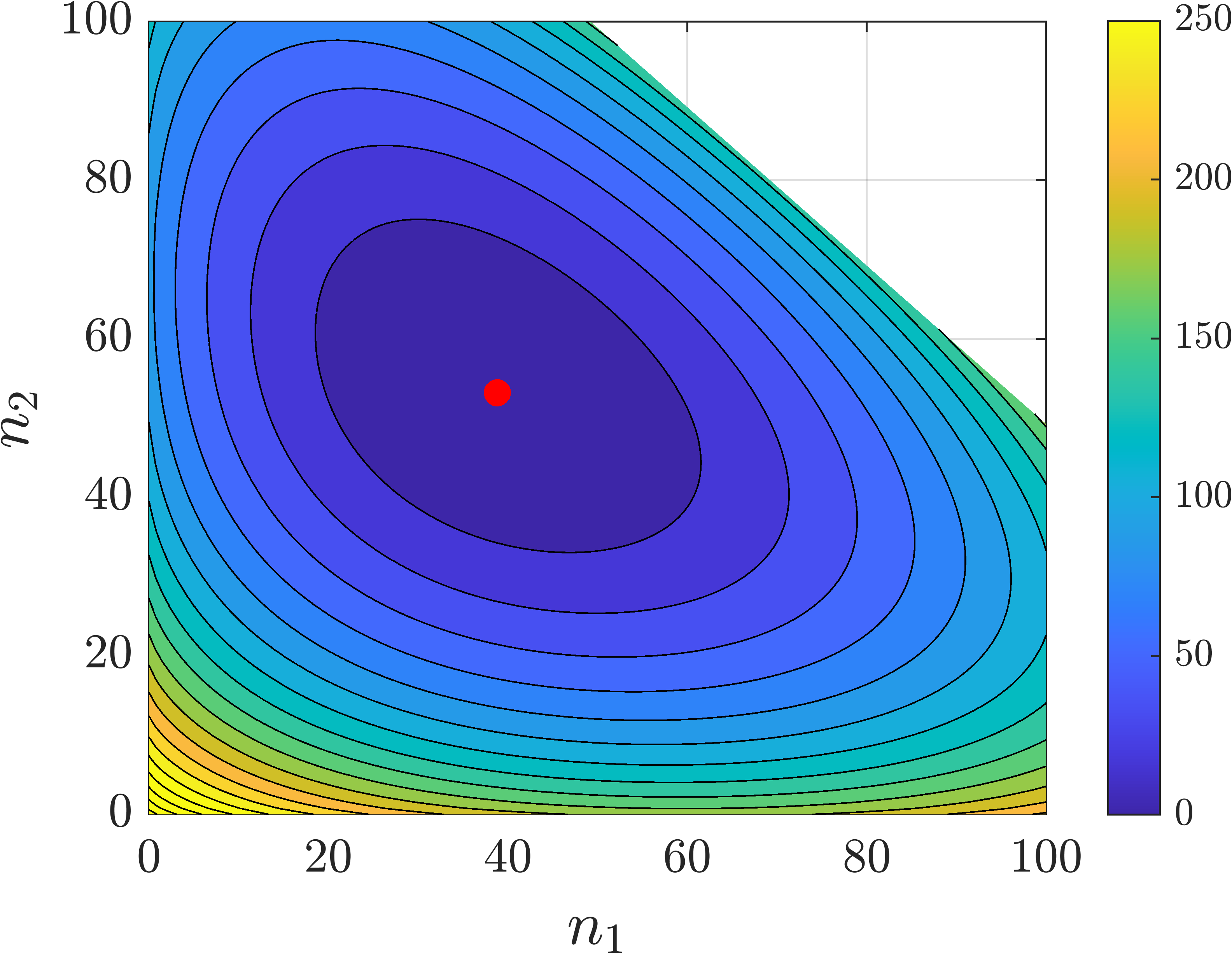}
			\caption{\scriptsize$l=200$, $\hat a=[2,3,2]$, $\hat b=[2,0,2]$}\label{fig:lyap:i}
		\end{subfigure}
		\caption{Comparison of Lyapunov functions.}\label{fig:lyapunovs}
	\end{center}
\end{figure}
The simulations were performed with $H(n_0)=150$. Figures \ref{fig:lyap:a}-\ref{fig:lyap:c} show the Lyapunov function $V^{(l,\hat a,\hat b)}$ for various choices of $\hat a$ and $\hat b$ with $l=25$ fixed. The second and third rows demonstrate the convergence characterized in \eqref{eq:Vconv}; figures \ref{fig:lyap:d}-\ref{fig:lyap:f} show $V^{(l,\hat a,\hat b)}$ for increasing $l$ values \ref{fig:lyap:g}-\ref{fig:lyap:i} shows $\sum_{i=1}^m\hat a_iV_i^{LTV}$ for the same increasing $l$ values.

\subsubsection*{Example 3}
Let us consider a compartmental system with $m=100$ compartments in the reduced state space. We assume that the transition rate functions are corresponding to Hill kinetics (modified intentionally to have different powers in the numerator and the denominator) and are of the form
\begin{equation}\label{eq:hill}
	\mathcal{K}_{ij}(n_i,c_j-n_j)=k_{ij}\frac{n_i^3(c_j-n_j)^3}{\qty\big(l+n_i^2)\qty\big(l+(c_j-n_j)^2)}
\end{equation}
with $l=350$. We assume that the only nonzero coefficients are
\begin{equation}
	\begin{aligned}
		&k_{i(i+1)}=20\quad k_{i(i+2)}=18\quad k_{i(i+3)}=16\quad k_{i(i+4)}=14\\
		&k_{i(i+5)}=12\quad k_{i(i+6)}=10\quad k_{i(i+7)}=8\quad k_{i(i+8)}=6
	\end{aligned}
\end{equation}
for $i=1,2,\dots,m$, where indices are understood as modulo $m$. Clearly this compartmental graph is strongly connected. Finally, we set capacities
\begin{equation}
	c_1=c_2=\dots=c_{50}=50\quad c_{51}=c_{52}=\dots=c_{100}=100
\end{equation}
Then the Lyapunov function \eqref{eq:lyap} takes the form
\begin{equation}
	\begin{aligned}
		V_{Hill}^{(l,3,2)}(n,\overline n)&=\sum_{i=1}^m\Bigg((\overline n_i-n_i)+3n_i\log\frac{n_i}{\overline n_i}+n_i\log\frac{\overline n_i^2+l}{n_i^2+l}\\
		&+2\sqrt l\qty\bigg(\atan\frac{\overline n_i}{\sqrt l}-\atan\frac{n_i}{\sqrt l})\Bigg).
	\end{aligned}
\end{equation}
We can also factorize as $\theta_i(r)=\hat\theta_i(r)\frac{r^2}{l+r^2}$, when \eqref{eq:lyap} becomes
\begin{equation}
	\begin{aligned}
		V_{Hill}^{(l,2,2)}(n,\overline n)=\sum_{i=1}^m\qty\Bigg(2n_i\log\frac{n_i}{\overline n_i}+n_i\log\frac{\overline n_i^2+l}{n_i^2+l}+2\sqrt l\qty\bigg(\atan\frac{\overline n_i}{\sqrt l}-\atan\frac{n_i}{\sqrt l})).
	\end{aligned}
\end{equation}
Figure \ref{fig:lyap:large} shows the time evolution of Lyapunov functions $V_{Hill}^{(l,3,2)}$, $V_{Hill}^{(l,2,2)}$ and $V^{LTV}$ and their time derivatives.
\begin{figure}[H]
	\centering
	\begin{subfigure}[b]{0.43\textwidth}
		\centering
		\includegraphics[width=\textwidth]{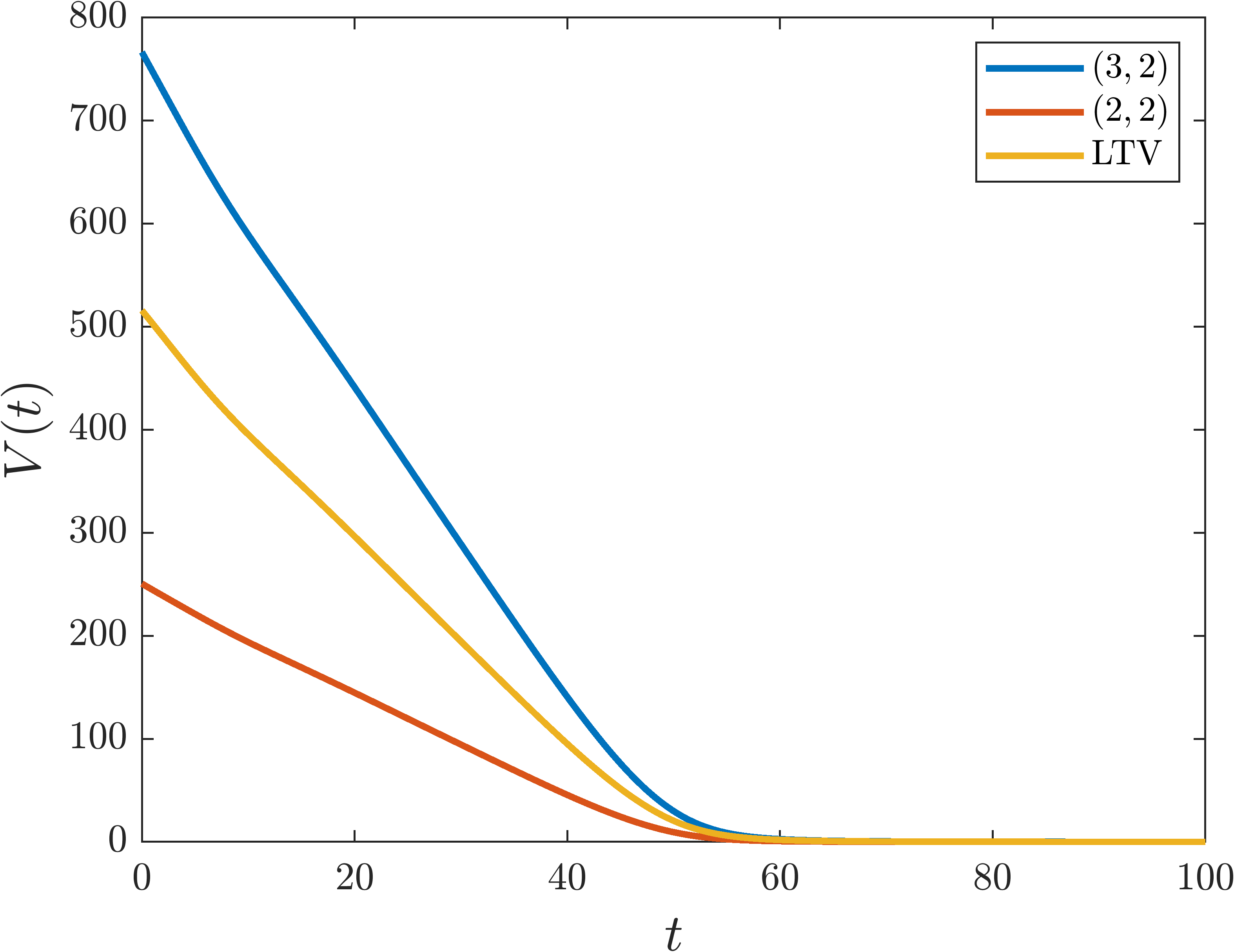}
		\caption{$V(t)$}
	\end{subfigure}
	\begin{subfigure}[b]{0.43\textwidth}
		\centering
		\includegraphics[width=\textwidth]{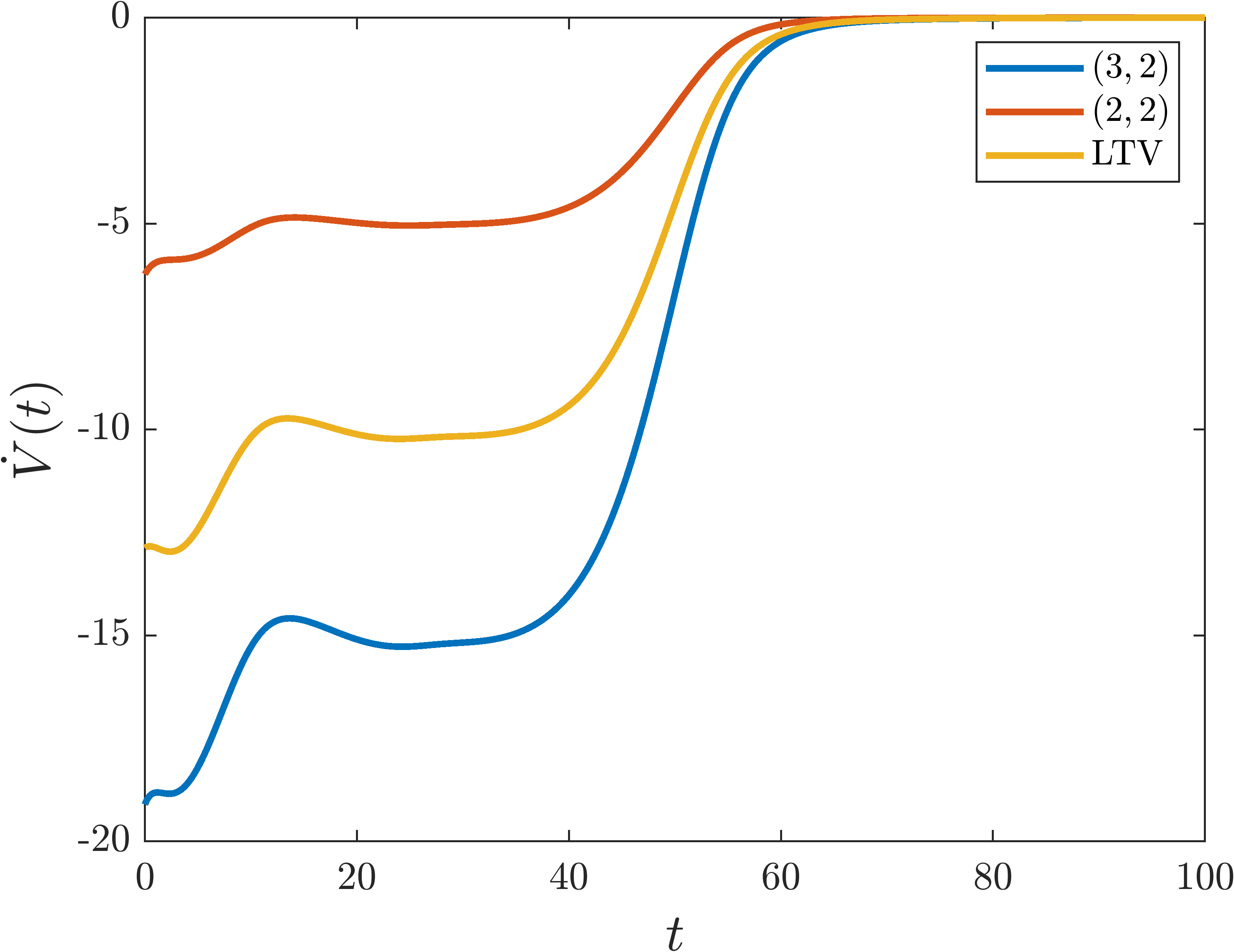}
		\caption{$\dot V(t)$}
	\end{subfigure}
	\caption{Time evolution of Lyapunov functions obtained from various factorizations of the transition rates}\label{fig:lyap:large}
\end{figure}

\begin{rem}
	In the above examples we restricted the factorizations to integer exponents so that we have real analytic transformations. However, the underlying dynamics is not changed through the factorizations and real analyticity is not directly used in the investigation of the ISS-Lyapunov function \eqref{eq:lyap}. Thus, as long as the factored $\hat k_{ij}(t)$ is piecewise locally Lipschitz (which holds after an arbitrarily short time in virtue of Remark \ref{rem:eps}), we can generalize \eqref{eq:lyap_kab} for other values as well; to be precise, we can use any $0<\hat a_i\le a_i$ and $0\le\hat b_i\le\hat a_i$ real numbers.\\

Next, focusing on the Hill kinetics in \eqref{eq:hill}, we note that while the denominator of the transformation $\theta_i(r)=\frac{r^3}{l+r^2}$ in \eqref{eq:hill} cannot be factorized we can rearrange the transformation as
\begin{equation}
	\theta_i(r)=\frac{r^3}{l+r^2}=\underbrace{\frac{r^{3-a_i}(l+r^{b_i})}{l+r^2}}_{\hat\theta_i(r)}\frac{r^{a_i}}{l+r^{b_i}}=\hat\theta_i(r)\frac{r^{a_i}}{l+r^{b_i}}
\end{equation}
where choosing $0<a_i\le3$ and $0\le b_i\le a_i$ ensures that the time-varying coefficient functions are piecewise locally Lipschitz. In this case the exact value of the integral in \eqref{eq:lyap} involves the generalized hypergeometric function and generally cannot be expressed in a closed form. However, in some special cases (such as $b_i=2$ above) we can calculate the integral explicitly; for example setting $a_i=1.5$ and $b_i=0.5$ yields
\begin{equation}
		V_{Hill}^{(l,1.5,0.5)}(n,\overline n)=\sum_{i=1}^m\qty\Bigg((\overline n_i-n_i)+\frac{3}{2}n_i\log\frac{n_i}{\overline n_i}+\qty\big(n_i-l^2)\log\frac{\sqrt{\overline n_i}+l}{\sqrt{n_i}+l}+l\qty\big(\sqrt{\overline n_i}-\sqrt{n_i})).
\end{equation}
\end{rem}

\section{Conclusions}\label{sec:con}
The dynamical properties of compartmental models with general time-varying transition rates were studied in this paper. The analysis is based on the CRN representation of such systems which has a transparent physical meaning by tracking the amounts of available objects and free spaces, respectively, in each compartment. It was shown that the deficiency of the kinetic model form is equal to the number of chordless cycles in the undirected reaction graph of the system. Furthermore, it was proved that time-varying generalized ribosome flows are persistent under mild regularity assumptions on the transition rates, and a wide set of reaction rates satisfying this assumption was characterized, containing well-known examples such as mass-action type rates. It was shown that the studied models can be embedded in at least two ways into the class of rate-controlled biochemical networks originally described in \cite{Chaves2005}. This embedding allows us to prove stability with entropy-like logarithmic Lyapunov functions known from the theory of CRNs. It was illustrated that the non-unique factorization of the rate functions gives rise to a whole family of various possible Lyapunov functions. Finally, periodic model behaviour was also studied, where we showed that trajectories with the same overall initial mass and periodic transition rates having the same period (but possibly different phase) converge to a unique periodic solution.
Besides the theoretical aspects, these improvements may efficiently support structural design or control of compartmental models, which is planned to be addressed in our future work.

\subsection*{Acknowledgements}
The work of M. A. V\'aghy has been supported by the \'UNKP-22-3-I-PPKE-72 National Excellence Program of the Ministry for Innovation and Technology from the source of the National Research, Development and Innovation Fund (NKFIH). The authors acknowledges the support of NKFIH through grants no. 131545.


\bibliographystyle{IEEEtran}
\bibliography{References}

\end{document}